\documentclass[11pt,reqno]{amsart}
\usepackage{tikz}
\usepackage{subfig}
\usepackage{epstopdf}
\usepackage{epsfig}
\usepackage{math}
\graphicspath{{./figures/}}
\usepackage{tikz}
\usetikzlibrary{shapes,arrows}
\usepackage{adjustbox}
\usepackage{rotating}
\usepackage{graphicx}

\tikzstyle{decision} = [diamond, draw, fill=blue!20,
    text width=4.5em, text badly centered, node distance=3cm, inner sep=0pt]
\tikzstyle{block} = [rectangle, draw, fill=blue!20,
    text width=5em, text centered, rounded corners, minimum height=4em]
\tikzstyle{line} = [draw, -latex']
\tikzstyle{cloud} = [draw, ellipse,fill=red!20, node distance=3cm,
    minimum height=2em]

\usetikzlibrary{positioning}
\tikzset{main node/.style={circle,fill=blue!20,draw,minimum size=1cm,inner sep=0pt},  }

\newcommand{\argmin}{\text{arg}\min}
\newcommand{\argmax}{\text{arg}\max}
\begin{document}
\title[]{Controlling conservation laws I: entropy--entropy flux}
\author[Li]{Wuchen Li}
\email{wuchen@mailbox.sc.edu}
\address{Department of Mathematics, University of South Carolina, Columbia }
\author[Liu]{Siting Liu}
\email{siting6@math.ucla.edu}
\address{Department of Mathematics, University of California, Los Angeles}
\author[Osher]{Stanley Osher}
\email{sjo@math.ucla.edu}
\address{Department of Mathematics, University of California, Los Angeles}
\newcommand{\vr}{\overrightarrow}
\newcommand{\wt}{\widetilde}
\newcommand{\dd}{\mathcal{\dagger}}
\newcommand{\ts}{\mathsf{T}}
\keywords{Conservation laws; Entropy--entropy flux pairs; Optimal transport; Mean-field games; Hamiltonian flows; Monotone schemes; Primal-dual algorithms.}
\thanks{W. Li thanks the start-up funding from the University of South Carolina,  and DEPSCoR Research Collaboration white paper ``Transport information geometric modeling of complex systems ". W. Li, S. Liu and S. Osher thank the funding from AFOSR MURI FA9550-18-1-0502 and ONR grants: N00014-18-1-2527, N00014-20-1-2093, and N00014-20-1-2787.}
\begin{abstract}
We study a class of variational problems for regularized conservation laws with Lax's entropy-entropy flux pairs. We first introduce a modified optimal transport space based on conservation laws with diffusion. Using this space, we demonstrate that conservation laws with diffusion are ``flux--gradient flows''. We next construct variational problems for these flows, for which we derive dual PDE systems for regularized conservation laws. Several examples, including traffic flow and Burgers' equation, are presented. Incorporating both primal-dual algorithms and monotone schemes, we successfully compute the control of conservation laws.  
\end{abstract}
\maketitle
\section{Introduction}
Regularized conservation laws\footnote{For simplicity, we omit the regularized throughout the paper.} are essential classes of dynamics in physics, materials sciences, and mathematical modeling, with applications to inverse problems and AI sampling problems \cite{Evans, Lax, VH}. Examples of conservation law equations include traffic flows, Burgers' equation, and compressible Navier-Stokes equations, etc.

Consider a system of initial value PDEs below.
\begin{equation}\label{PDE}
\partial_t u(t,x)+\mathcal{B}(u(t,x))=\beta \mathcal{C}(u(t,x)),\quad u(0,x)=u_0,
\end{equation}
where $x\in \Omega\subset\mathbb{R}^n$, $u\colon [0,\infty)\times\Omega\rightarrow\mathbb{R}^d$ is a unknown vector function, $u_0$ is a given initial condition, $\mathcal{B}\colon C^{\infty}(\Omega; \mathbb{R}^d)\rightarrow C^{\infty}(\Omega; \mathbb{R}^d)$ is a ``conservative'' differential operator, $\mathcal{C}\colon C^{\infty}(\Omega; \mathbb{R}^d)\rightarrow C^{\infty}(\Omega; \mathbb{R}^d)$ is a dissipative differential (diffusion) operator and $\beta\geq 0$ is a diffusion constant. For simplicity of presentation, we assume that $\Omega$ is a compact convex set, and the PDE \eqref{PDE} has periodic boundary conditions. E.g., $\Omega=\mathbb{T}^n$, where $\mathbb{T}^n$ is a $n$--dimensional torus.

In this paper, we introduce variational problems related to the PDE \eqref{PDE}. They generalize mean-field information dynamics \cite{LiG1, LiHess}; see Figure \ref{flowchart}.
\begin{figure}
{\includegraphics[scale=0.24]{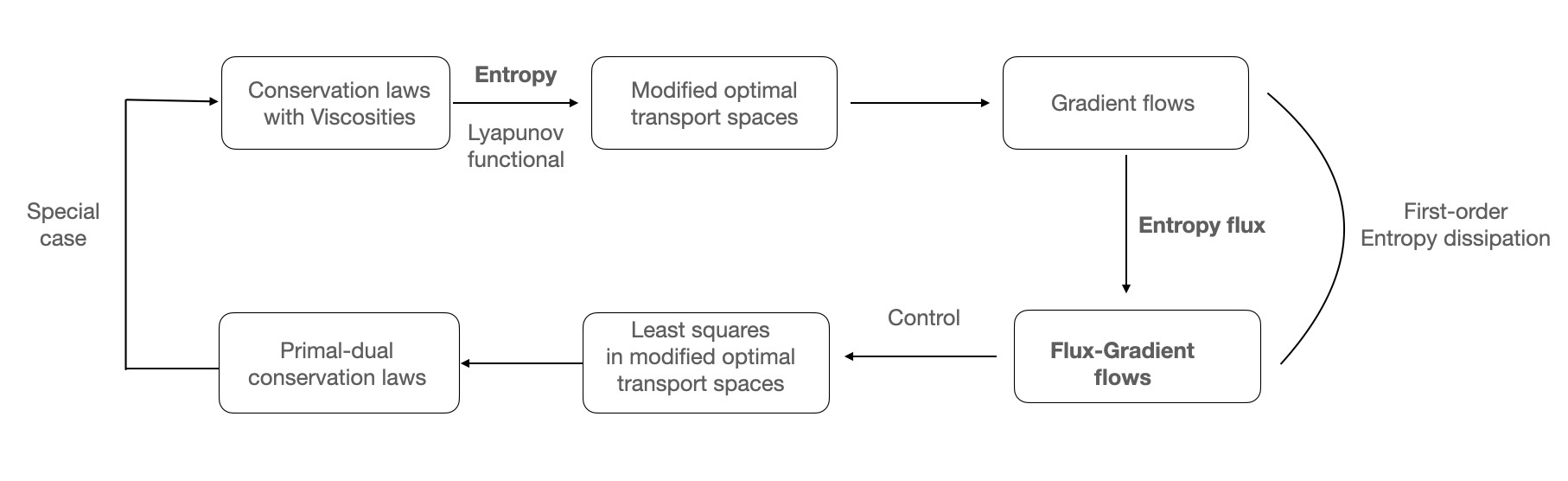}}
\caption{We study variational problems of conservation laws in entropy-entropy flux pairs induced metric spaces.}
\label{flowchart}
\end{figure}
Here we shall design suitable {\em modified optimal transport spaces} for PDE \eqref{PDE}, namely mean-field information metric spaces. In these metric spaces, we demonstrate that the PDE \eqref{PDE} has a dissipative variational structure. And we name PDE \eqref{PDE} {\em flux--gradient flows}. We then design a control problem of PDE \eqref{PDE} in metric space. By finding the critical point of control problem, we derive a dual PDE system for equation \eqref{PDE}. The primal-dual pair of PDE system satisfies a Hamiltonian flow in the metric space, where the Hamiltonian functional depends on the flux function and ``kinetic energy''. Several numerical examples, including traffic flow and Burgers' equation, are presented.

The main results are sketched below. 

\noindent\textbf{Assumption 1:} Suppose that there exists a function $G\colon \mathbb{R}^d\rightarrow\mathbb{R}$, such that
\begin{equation*}
\int_\Omega  G'(u)\cdot\mathcal{B}(u)  dx= 0.
\end{equation*}
\noindent\textbf{Assumption 2:} Suppose that there exists a ``symmetric nonnegative definite'' operator $L_\mathcal{C}(u)\colon C^\infty(\Omega)\rightarrow C^{\infty}(\Omega)$, such that
\begin{equation*}
\int_\Omega G'(u)\cdot \mathcal{C}(u)dx=-\int_\Omega (G'(u), L_\mathcal{C}(u)G'(u))dx\leq 0.
\end{equation*}
 Following \cite{Lax}, under assumption 1 and 2, we demonstrate that 
\begin{equation*}
\mathcal{G}(u(t,\cdot))=\int_\Omega G(u(t,x))dx,
\end{equation*}
forms a Lyapunov functional for PDE \eqref{PDE}. In other words, along PDE \eqref{PDE}, the time derivative of $\mathcal{G}(u)$ is non-positive:
\begin{equation*}
\begin{split}
\frac{d}{dt}\mathcal{G}(u(t,\cdot))=&\int_\Omega G'(u)\partial_tu dx=\int_\Omega G'(u)\cdot (-\mathcal{B}(u)+\beta \mathcal{C}(u)) dx\\
=&\beta\int_\Omega G'(u) \cdot\mathcal{C}(u) dx=-\beta\int_\Omega (G'(u), L_\mathcal{C}(u)G'(u))dx\leq 0.
\end{split}
\end{equation*}

Based on the Lyapunov functional $\mathcal{G}$, we design optimal control problems of PDE \eqref{PDE}. In detail: 
let $d=1$,  
\begin{equation*}
\mathcal{B}(u)=\nabla\cdot f(u), \quad \mathcal{C}(u)=\nabla\cdot(A(u)\nabla u),\quad \mathcal{L}_{\mathcal{C}}(u):=-\nabla\cdot(A(u)G''(u)^{-1}\nabla),
\end{equation*}
where $f\colon \mathbb{R}\rightarrow\mathbb{R}^n$ is a flux function and $A, AG''^{-1}\colon\Omega\rightarrow\mathbb{R}^{n\times n}$ are both symmetric non-negative definite  matrix functions.
Given a suitable potential functional $\mathcal{F}\colon C^{\infty}(\Omega;\mathbb{R})\rightarrow\mathbb{R}$ and a terminal functional $\mathcal{H}\colon C^{\infty}(\Omega;\mathbb{R})\rightarrow\mathbb{R}$, consider a variational problem 
\begin{subequations}\label{main_variation}
\begin{equation}\label{mv1}
\inf_{u, v, u_1}~\int_0^1\Big[\int_\Omega \frac{1}{2}\big(v, A(u)G''(u)^{-1}v\big)dx-\mathcal{F}(u)\Big]dt+\mathcal{H}(u_1),
\end{equation}
where the infimum is taken among variables $v\colon [0,1]\times \Omega\rightarrow\mathbb{R}^n$, $u\colon [0,1]\times\Omega\rightarrow\mathbb{R}$, and $u_1\colon\Omega\rightarrow\mathbb{R}$ satisfying 
\begin{equation}\label{mv2}
\begin{split}
\partial_t u+\nabla\cdot f(u)+\nabla\cdot(A(u)G''(u)^{-1}v)=\beta \nabla\cdot(A(u)\nabla u),\quad u(0,x)=u_0(x).
\end{split}
\end{equation}
\end{subequations}

By Proposition \ref{MFH}, we derive a critical point system of the optimal control problem \eqref{main_variation}. Define a Lagrange multiplier function $\Phi\colon [0,1]\times\Omega\rightarrow\mathbb{R}$. Then $v=\nabla\Phi$, and $(u,\Phi)$ satisfies a pair of PDEs:
\begin{subequations}\label{DS}
\begin{equation}\label{DS1}
\left\{\begin{aligned}
&\partial_t u+\mathcal{B}(u)-L_\mathcal{C}(u)(\Phi)-\beta \mathcal{C}(u)=0,\\
&\partial_t\Phi+\frac{\delta}{\delta u}\int_\Omega \Big(\frac{1}{2}(\Phi, L_\mathcal{C}(u)\Phi)-(\mathcal{B}(u), \Phi)+\beta (\Phi, \mathcal{C}(u))\Big)dx+\frac{\delta}{\delta u}\mathcal{F}(u)=0,
\end{aligned}\right.
\end{equation}
with both initial and terminal time boundary conditions
\begin{equation}\label{DS2}
u(0,x)=u_0(x),\quad \frac{\delta}{\delta u_1(x)}\mathcal{H}(u_1)+\Phi(1,x)=0.
\end{equation}
\end{subequations}
In above formulations, $\frac{\delta}{\delta u}$ is the $L^2$ first variation operator w.r.t. variable $u$. Here we aim to study the variational problem \eqref{main_variation} and its critical point system \eqref{DS} in the following two aspects. 
\begin{itemize}
\item[(i)] Modeling: We design and control models for conservation law dynamics. A typical example is that we control the kinetic energy of the density under traffic flow.  
\item[(ii)] Computation: By choosing some special potential and terminal energies, the minimizer of the mean--field information control problem \eqref{main_variation} is consistent with the initial value problem \eqref{PDE}. This leads to stable and convergent primal-dual algorithms for initial value conservation laws. 
\end{itemize}

In the literature, there are joint studies towards dynamics in mean-field information metric spaces; see \cite{AGS, EM, BB, BM, GNT, AM, otto2001, VH, Villani2009_optimal} and many references therein. Several types of dynamics in these metric spaces have been studied in recent decades. First, gradient flows have been systematically studied in \cite{AGS, C1, O1}. Next, Hamiltonian flows with generalizations to differential games have been investigated by \cite{MFGC, LL,MRPP}. A particular type of Hamiltonian flows, namely Schr{\"o}dinger bridge systems, and their mean-field generalizations, are widely studied in \cite{BCGL, CGP, DV, CCT,CL, LeL}. Compared to previous works, we focus on conservative-dissipative equations, which are non--gradient flow systems. In particular, we connect the conservation law equations with generalized optimal transport variational structures. Concretely, we study equation \eqref{PDE} associated with a least-square type control problem in the designed metric space. We remark that the control problem in \eqref{main_variation} connects with, but is different from, the ones in \cite{YB, BM}. In detail, \cite{BM} designed variational problems in term of entropy functional $-\mathcal{F}(u)=\mathcal{G}(u)=\int_\Omega G(u)dx$, which can solve the initial value conservation laws. And \cite{BM} also enforces $v=0$ in the control problem \eqref{main_variation}. In contrast, we control ``kinetic energy'' in the modified optimal transport space generated by the entropy--entropy flux pairs. Following this approach, the PDE pair \eqref{DS} can be used to control conservation laws. 

The paper is organized below. In section \ref{section2}, we provide two motivating examples of variational problem \eqref{main_variation}. In section \ref{section3} and \ref{section4}, we present the main result of this paper. We first present the variational formulation of equation \eqref{PDE} in modified optimal transport space. We next derive a variational problem and a pair of PDEs for conservation laws. In section \ref{primal-dual}, we design primal-dual algorithms to numerically solve the optimal control problems of conservation laws.

\section{Motivation}\label{section2}
In this section, we provide two examples of conservative-dissipative equations \eqref{PDE}. In these examples, we demonstrate variational problems \eqref{main_variation} and PDE pairs \eqref{DS} with their physical and modeling explanations.

\subsection{Viscous Burgers' equation}
Consider a one dimensional viscous Burgers' equation
\begin{equation*}
\partial_t u(t,x)+\partial_x\frac{u(t,x)^2}{2}=\beta\partial_{xx} u(t,x).
\end{equation*}
Here $\partial_x$, $\partial_{xx}$ are the first and the second derivatives w.r.t. $x$, and the unknown variable is $u\colon \mathbb{R}_+\times \Omega\rightarrow \mathbb{R}$. In this case, the conservative and dissipative operators satisfy
\begin{equation*}
\mathcal{B}(u)=\partial_x(\frac{u^2}{2}), \qquad\mathcal{C}(u)=\partial_{xx} u.
\end{equation*}
Define a function $G\colon \mathbb{R}^1\rightarrow\mathbb{R}^1$ by
\begin{equation*}
G(u)=\frac{u^2}{2}, \quad G'(u)=u.
\end{equation*}
Then assumption 1 is satisfied since
\begin{equation*}
\int_\Omega  G'(u)\cdot\mathcal{B}(u) dx=\int_\Omega \partial_x(\frac{u^2}{2})\cdot u dx=\int_\Omega u^2\cdot\partial_x u dx =\int_\Omega \partial_x(\frac{u^3}{3})dx=0.
\end{equation*}
Denote a ``symmetric nonnegative definite operator" by 
\begin{equation*}
L_{\mathcal{C}}(u)=-\partial_{xx}.
\end{equation*}
 Assumption 2 is satisfied since
\begin{equation*}
\int_\Omega G'(u)\cdot\mathcal{C}(u)dx=\int_\Omega u \cdot \partial_{xx} u dx=-\int_\Omega |\partial_x u|^2 dx,\end{equation*}
where the second equality is from the integration by parts formula.

\noindent{\em Formulations:} Variational problem \eqref{main_variation} forms an optimal control problem for viscous Burgers' equation. Consider 
\begin{subequations}\label{main_variation_example_1}
\begin{equation}
\inf_{u, v, u_1}~\int_0^1\Big[\int_\Omega \frac{1}{2}|v(t,x)|^2 dx-\mathcal{F}(u)\Big]dt+\mathcal{H}(u_1),
\end{equation}
such that
\begin{equation}
\partial_t u(t,x)+\partial_x(\frac{u(t,x)^2}{2})+\partial_{x}v(t,x)=\beta \partial_{xx} u(t,x),\quad u(0,x)=u_0(x).
\end{equation}
\end{subequations}
Here, the minimizer system of variational problem \eqref{main_variation_example_1} satisfies a pair of PDEs: $v(t,x)=\nabla\Phi(t,x)$ and 
\begin{equation*}
\left\{\begin{aligned}
&\partial_t u(t,x)+\partial_x(\frac{u(t,x)^2}{2})+\partial_{xx}\Phi(t,x)=\beta \partial_{xx} u(t,x),\\
&\partial_t\Phi(t,x)+(u(t,x), \partial_x \Phi(t,x))+\frac{\delta}{\delta u}\mathcal{F}(u)(t,x)=-\beta\partial_{xx} \Phi(t,x).
\end{aligned}\right.
\end{equation*}

\noindent{\em Velocity control.} 
The unknown variable $u$ in Burgers' equation describes the velocity filed of the fluid over time. In the variational problem \eqref{main_variation_example_1}, we design a potential field $\Phi$ to control the evolution of velocity field $u$. Here, the background dynamics of $u$ is the classical initial value viscous Burgers' equation. The designed variational problem is to control a velocity field under suitable running and terminal costs.

\subsection{Traffic flow equation}\label{sec:traffic_flow}
Consider a one dimensional traffic flow equation
\begin{equation*}
\partial_t u(t,x)+\nabla\cdot\big(u(t,x)(1-u(t,x))\big)=\beta\Delta u(t,x).
\end{equation*}
Here $\nabla\cdot$, $\Delta$ are divergence, Laplacian operators w.r.t $x$, respectively, and the unknown variable is  $u\colon \mathbb{R}_+\times \Omega\rightarrow \mathbb{R}_+$. In this case, the conservative and the dissipative operators satisfy
\begin{equation*}
\mathcal{B}(u)=\nabla\cdot(u(1-u)), \qquad\mathcal{C}(u)=\Delta u.
\end{equation*}
Define a function $G\colon \mathbb{R}_+^1\rightarrow\mathbb{R}_+^1$ by
\begin{equation*}
G(u)=u\log u-u,\quad G'(u)=\log u.
\end{equation*}
Then assumption 1 is satisfied since
\begin{equation*}
\int_\Omega G'(u)\cdot \mathcal{B}(u) dx=\int_\Omega \log u\cdot \nabla\cdot(u (1-u)) dx=\int_\Omega \nabla\cdot\Psi(u)dx=0,
\end{equation*}
where $\Psi(u)=\int_0^u (1-2z)\log zdz$.
Denote a ``symmetric nonnegative definite operator" by 
\begin{equation*}
L_{\mathcal{C}}(u)=-\nabla\cdot(u \nabla).
\end{equation*}
Assumption 2 is satisfied since
\begin{equation*}
\begin{split}
\int_\Omega G'(u)\cdot \mathcal{C}(u) dx
=&\int_\Omega \log u \cdot \Delta u dx\\
=&\int_\Omega (\nabla \log u, \nabla u)dx\\
=&-\int_\Omega \|\nabla \log  u\|^2 u dx\\
=&-\int_\Omega \big(\log u, -\nabla\cdot(u\nabla \log u)\big) dx,
\end{split}
\end{equation*}
where the second equality is from the integration by parts formula and the third equality holds by the fact that $\frac{\nabla u}{u}=\nabla \log u$, i.e., $\nabla u=u\nabla\log u$. 

\noindent{\em Formulations:} Variational problem \eqref{main_variation} forms an optimal control problem for traffic flows. In detail, consider 
\begin{subequations}\label{main_variation_example_2}
\begin{equation}
\inf_{u,v,v_1}~\int_0^1\Big[\int_\Omega \frac{1}{2}\|v(t,x)\|^2u(t,x) dx-\mathcal{F}(u)\Big]dt+\mathcal{H}(u_1),
\end{equation}
such that
\begin{equation*}
0\leq u(t,x)\leq 1, \quad \textrm{for all $t\in [0,1]$},
\end{equation*}
and
\begin{equation}
\partial_t u(t,x)+\nabla\cdot\big(u(t,x)(1-u(t,x))\big)+\nabla\cdot\big(u(t,x)v(t,x)\big)=\beta \Delta u(t,x),\quad u(0,x)=u_0(x).
\end{equation}
\end{subequations}
Here, the minimizer system of variational problem \eqref{main_variation_example_1} satisfies a pair of PDEs. When $u\in (0,1)$, there exists a scalar function $\Phi$, such that $v(t,x)=\nabla\Phi(t,x)$ and 
\begin{equation*}
\left\{\begin{aligned}
&\partial_t u(t,x)+\nabla\cdot\big(u(t,x)(1-u(t,x))\big)+\nabla\cdot\big(u(t,x)\nabla\Phi(t,x)\big)=\beta \Delta u(t,x),\\
&\partial_t\Phi(t,x)+\big(1-2u(t,x), \nabla \Phi(t,x)\big)+\frac{1}{2}\|\nabla\Phi(t,x)\|^2+\frac{\delta}{\delta u(t,x)}\mathcal{F}(u)=-\beta\Delta \Phi(t,x).
\end{aligned}\right.
\end{equation*}
\noindent{\em Position control.} The unknown variable $u$ in traffic flows represents the density function of cars (particles) in a given spatial domain. Here, the background dynamics of $u$ is the classical traffic flow. The control variable is the velocity for enforcing each car's velocity in addition to its background traffic flow dynamics. The goal is to control the ``total enforced kinetic energy'' of all cars, in which individual cars can determine their velocities through both noises and traffic flow interactions.

We shall demonstrate that variational problems \eqref{main_variation} can be formulated in term of general conservation law equations associated with entropy-entropy flux pairs. 

\section{Entropy-entropy flux and flux--gradient flows}\label{section3}
In this section, we first recall Lax's entropy-entropy flux pairs for conservation law equations; see \cite{Evans, Lax}. Here the entropy-entropy flux pair is used to construct a Lyapunov functional for PDE \eqref{PDE}. Using this Lyapunov functional with the dissipative operator, we next review and formulate both metric spaces and gradient flows; see \cite{AGS, C1,AM}. Combining these facts with flux functions, we name PDE \eqref{PDE} {\em flux--gradient flows} in metric spaces. In later on sections, we shall demonstrate that the flux--gradient flow formulation is useful in designing control problems of conservation laws. 
\subsection{Entropy-Entropy flux pairs and Lyapunov functionals}
 Consider 
\begin{equation*}
u\colon \mathbb{R}_+\times \Omega\rightarrow\mathbb{R}^1,\quad \mathcal{B}(u)=\nabla\cdot f(u),\quad \mathcal{C}(u)=\nabla\cdot(A(u)\nabla u).
\end{equation*}
In this case, equation \eqref{PDE} satisfies
\begin{equation}\label{cld}
\partial_t u(t,x)+\nabla\cdot f(u(t,x))=\beta \nabla\cdot(A(u)\nabla u),
\end{equation}
where $u\colon \mathbb{R}^n\rightarrow\mathbb{R}$ is a scalar function, $f=(f_1, \cdots, f_n)$ is a flux vector function, and $A=(A_{ij})_{1\leq i,j\leq n}\in\mathbb{R}^{n\times n}$ is a semi-positive definite matrix function. 
 If $\beta=0$, equation \eqref{cld} is a scalar conservation law equation
\begin{equation}\label{cl}
\partial_t u(t,x)+\nabla\cdot f(u(t,x))=0.
\end{equation}
\begin{definition}[Entropy-entropy flux pair condition \cite{Lax}]
We call $(G,\Psi)$ an entropy-entropy flux pair for the conservation law \eqref{cl} if there exists a convex function $G\colon \mathbb{R}\rightarrow\mathbb{R}$, and $\Psi\colon \mathbb{R}\rightarrow\mathbb{R}^n$, such that
\begin{equation*}
\Psi'(u)=f'(u)G'(u).
\end{equation*}
\end{definition}
\begin{remark}
We remark that the entropy-entropy flux condition is trivial for scalar conservation laws. Every differentiable convex function $G$ satisfies this condition. For simplicity of presentation, we only focus on the scalar case, and leave the study of conservation law systems in future work.  
\end{remark}
In fact, Lax's entropy--entropy flux pair introduces a class of functionals, which can be used as Lyapunov functionals for PDE \eqref{PDE}.  Denote
\begin{equation*}
\mathcal{G}(u)=\int_\Omega G(u)dx.
\end{equation*}

For assumption 1,
\begin{equation*}
\begin{split}
\int_\Omega G'(u)\cdot \mathcal{B}(u) dx=&\int_\Omega G'(u) \nabla\cdot f(u)dx\\
=&\int_\Omega G'(u) (f'(u), \nabla u) dx\\
=&\int_\Omega (\Psi'(u), \nabla u) dx\\
=&\int_\Omega \nabla\cdot\Psi(u)dx=0,
\end{split}
\end{equation*}
where we apply the fact that $f'(u)G'(u)=\Psi'(u)$ and $\int_\Omega \nabla\cdot\Psi(u)dx=0$ in the last equality. 

For assumption 2,
\begin{equation*}
\begin{split}
\int_\Omega G'(u)\cdot \mathcal{C}(u) dx
=& \int_\Omega G'(u)\nabla\cdot(A(u)\nabla u) dx\\
=&-\int_\Omega \big(\nabla G'(u), A(u)\nabla u\big)dx\\
=&-\int_\Omega\big(\nabla G'(u), A(u)G''(u)^{-1}\nabla G'(u)\big) dx,
\end{split}
\end{equation*}
where we apply $\nabla G'(u)=G''(u)\nabla u$ in the last equality. We observe a fact that $A(u)G''(u)^{-1}\succeq 0$, since $A$ is nonnegative definite and $G''(u)>0$.
Hence we know that 
\begin{equation*}
\begin{split}
\partial_t\mathcal{G}(u)=&\int_\Omega G'(u)\cdot\partial_t u dx\\
=&-\int_\Omega G'(u)\cdot \mathcal{B}(u)dx+\beta\int_\Omega G'(u)\cdot\mathcal{C}(u)dx\\
=&-\beta\int_\Omega\big(\nabla G'(u), A(u)G''(u)^{-1}\nabla G'(u)\big) dx\leq 0. 
\end{split}
\end{equation*}
This implies that $\mathcal{G}(u)$ is a Lyapunov functional for PDE \eqref{cld}.

\subsection{Metric spaces and gradient flows}
We next provide a condition to define a metric space for the unknown variable $u$. Here, the metric space connects  Lyapunov functionals with dissipative operators through gradient descent flows.  

\begin{definition}[Entropy-entropy flux-metric condition]\label{def2}
We call $(G,\Psi)$ an entropy-entropy flux pair-metric for equation \eqref{cld} if there exists a convex function $G\colon \mathbb{R}\rightarrow\mathbb{R}$, and $\Psi\colon \mathbb{R}\rightarrow\mathbb{R}^n$, such that
\begin{equation*}
\Psi'(u)=f'(u)G'(u),\quad \textrm{$A(u)G''(u)^{-1}$ is a semi positive definite symmetric matrix function}.
\end{equation*}
\end{definition}
Under the entropy-entropy flux-metric condition, assumption 2 implies a metric operator below.
Define the space of function $u$ by
 \begin{equation*}
 \mathcal{M}=\Big\{u\in C^{\infty}(\Omega)\colon \int_\Omega u(x)dx=\mathrm{constant}\Big\}.
 \end{equation*}
The tangent space of $\mathcal{M}(u)$ at point $u$ is defined by
\begin{equation*}
T_u\mathcal{M}=\Big\{\sigma\in C^{\infty}(\Omega)\colon \int_\Omega \sigma(x)dx=0\Big\}.
\end{equation*}
Denote an elliptic operator $L_\mathcal{C}\colon C^{\infty}(\Omega)\rightarrow C^{\infty}(\mathbb{R})$ by
\begin{equation*}
L_\mathcal{C}(u)=-\nabla\cdot(A(u)G''(u)^{-1}\nabla).
\end{equation*}
We are ready to define the metric inner product for the dissipation operator. 
\begin{definition}[Metric]\label{metrics}
  The inner product $\g(u)\colon
  {T_u}\mathcal{M}\times{T_u}\mathcal{M}\rightarrow\mathbb{R}$ is given below.
 \begin{equation*}\begin{split}
    \g(u)(\sigma_1, \sigma_2)  =&\int_\Omega (\Phi_1, L_{\mathcal{C}}(u)\Phi_2)dx\\
   =&-\int_\Omega \Phi_1 \nabla\cdot(A(u)G''(u)^{-1}\nabla\Phi_2) dx\\
  =&  \int_\Omega (\nabla\Phi_1, A(u)G''(u)^{-1}\nabla\Phi_2) dx\\
  =&\int_\Omega \sigma_1\Phi_2 dx=\int_\Omega \sigma_2\Phi_1dx,
  \end{split}
\end{equation*}
where $\Phi_i\in C^{\infty}(\Omega)$ satisfies \begin{equation*}
\sigma_i=-\nabla\cdot(A(u)G''(u)^{-1}\nabla\Phi_i),\quad i=1,2.
\end{equation*}
\end{definition}
From now on, we call $(\mathcal{M}, \g)$ the metric space. We next review and present the gradient decent flow in metric space $(\mathcal{M}, \g)$. Here the dissipative part of PDE \eqref{PDE} comes from the gradient flow of the proposed Lyapunov (entropy) functional.
\begin{proposition}[Gradient flow]\label{gd}
Given an energy functional $\mathcal{E}\colon \mathcal{M}\rightarrow\mathbb{R}$, the gradient flow of $\mathcal{F}$ in $(\mathcal{M}, \g)$ satisfies
\begin{equation*}
\partial_t u=\nabla\cdot(A(u)G''(u)^{-1}\nabla \frac{\delta}{\delta u}\mathcal{E}(u)).
\end{equation*}
If $\mathcal{E}(u)= \mathcal{G}(u)=\int_\Omega G(u)dx$, then the above gradient flow satisfies
\begin{equation*}
\partial_t u=\nabla\cdot(A(u)\nabla u).
\end{equation*}
\end{proposition}
\begin{proof}
The derivation of gradient flow follows from the definition. For the completeness of this paper, we still present the proof here.
The gradient operator $\mathrm{grad}\mathcal{E}\in T_u\mathcal{M}$ is defined by
\begin{equation*}
\g(u)\big(\mathrm{grad}\mathcal{E}(u),\sigma\big)=\int_\Omega (\frac{\delta}{\delta u}\mathcal{E}(u),\sigma)dx,
\end{equation*}
for any $\sigma\in T_u\mathcal{M}$. Let $\sigma=L_{\mathcal{C}}(u)(\Phi)=-\nabla\cdot(A(u)G''(u)^{-1}\nabla\Phi)$, the above definition forms
\begin{equation*}
\begin{split}
\g(u)\big(\mathrm{grad}\mathcal{E}(u),\sigma\big)=&\int_\Omega (\mathrm{grad}\mathcal{E}(u),\Phi)dx\\
=&\int_\Omega (\frac{\delta}{\delta u}\mathcal{E}(u), L_\mathcal{C}(u)(\Phi))dx\\
=&\int_\Omega (L_\mathcal{C}(u)(\frac{\delta}{\delta u}\mathcal{E}(u)), \Phi)dx.
\end{split}
\end{equation*}
Since the above equality holds for any smooth function $\Phi$, we let
\begin{equation*}
\mathrm{grad}\mathcal{E}(u)=L_{\mathcal{C}}(u)\big(\frac{\delta}{\delta u}\mathcal{E}(u)\big)=-\nabla\cdot\Big(A(u)G''(u)^{-1}\nabla\frac{\delta}{\delta u}\mathcal{E}(u)\Big).
\end{equation*}
Hence the gradient decent flow satisfies
\begin{equation*}
\partial_tu=-\mathrm{grad}\mathcal{E}(u)=\nabla\cdot\Big(A(u)G''(u)^{-1}\nabla\frac{\delta}{\delta u}\mathcal{E}(u)\Big).
\end{equation*}
If $\mathcal{E}(u)=\mathcal{G}(u)$, then 
\begin{equation*}
\begin{split}
\partial_t u=&\nabla\cdot(A(u)G''(u)^{-1}\nabla \frac{\delta}{\delta u}\mathcal{G}(u))\\
=&\nabla\cdot(A(u)G''(u)^{-1}\nabla G'(u))\\
=&\nabla\cdot(A(u)G''(u)^{-1}G''(u)\nabla u)\\
=&\nabla\cdot(A(u)\nabla u).
\end{split}
\end{equation*}
\end{proof}
\begin{remark}
An example of metric $g$ is the Wasserstein-2 metric, which has been widely studied in optimal transport literature \cite{AGS, Villani2009_optimal}. In other words, let $G(u)=u\log u-u$, $A(u)=\mathbb{I}$, then $L_\mathcal{C}(u)=-\nabla \cdot(u\nabla)$. In this case, the heat equation is the gradient flow of entropy functional $\int_\Omega u\log u- u dx$ in Wasserstein-2 space. Recently, generalized optimal transport metrics and nonlinear diffusions have been widely studied in \cite{C1, O1, AM1}. 
\end{remark}
\subsection{Flux--gradient flows}
In summary, we illustrate the relation among entropy-entropy flux pairs, gradient flows and PDE \eqref{cld}. On the one hand, the entropy-entropy flux pairs introduce a Lyapunov functional, along which the entropy-entropy flux flow is non positive. In addition, $\frac{dG}{dt}+\nabla\cdot\Psi\leq 0$. This entropy condition \cite{Lax} picks out the unique physical weak solution for inviscid conservation law \eqref{cl}. On the other hand, both Lyapunov functional and dissipative operator define a metric space, under which the dissipative operator forms the gradient flow. These facts imply a formulation for PDE \eqref{cld}. We call it {\em flux--gradient flows}.

In detail, equation \eqref{cld} can be written below. 
\begin{equation}\label{pgdd}
\partial_t u+\nabla\cdot f(u)=\beta\nabla\cdot(A(u)G''(u)^{-1}\nabla \frac{\delta}{\delta u}\mathcal{G}(u)),
\end{equation}
where the flux function $f$ satisfies 
\begin{equation*}
\int_\Omega f(u)\cdot\nabla \frac{\delta}{\delta u}\mathcal{G}(u) dx=0.
\end{equation*}

The above formulation of the PDE \eqref{cld} is a gradient flow equation added with the flux function. We can consider its general formulation. We keep the metric operator invariant and replace the Lyapunov functional $\mathcal{G}(u)$ by a general energy functional $\mathcal{E}(u)$.

\begin{definition}[Flux--gradient flow]\label{def5}
Given an energy functional $\mathcal{E}\colon\mathcal{M}\rightarrow\mathbb{R}$, consider a class of PDE  
\begin{equation}\label{pgd}
\partial_t u+\nabla\cdot(f_1(x,u))=\beta\nabla\cdot(A(u)G''(u)^{-1}\nabla \frac{\delta}{\delta u}\mathcal{E}(u)),
\end{equation}
where $f_1\colon \Omega\times\mathbb{R}^1\rightarrow\mathbb{R}^n$ is a flux function satisfying 
\begin{equation}\label{pgdcond}
\int_\Omega f_1(x,u)\cdot\nabla\frac{\delta}{\delta u(x)}\mathcal{E}(u)dx=0.
\end{equation}
If $\mathcal{E}(u)= \mathcal{G}(u)= \int_\Omega G(u)dx$ and $f_1(x,u)=f(u)$, then equation \eqref{pgd} forms PDE \eqref{cld}.
\end{definition}
Equations in formulation \eqref{pgd} provide a class of conservative--dissipative dynamics.
\begin{proposition}[Entropy-entropy flux-production]
Energy functional $\mathcal{E}(u)$ is a Lyapunov functional for PDE \eqref{pgd}. In other words, the following dissipation result holds.
Suppose $u(t,x)$ is the solution of equation \eqref{pgd}, then
\begin{equation*}
\frac{d}{dt}\mathcal{E}(u(t,\cdot))=-\beta\mathcal{I}_{\mathcal{E}}(u(t,\cdot))\leq 0,
\end{equation*}
where the functional $\mathcal{I}_{\mathcal{E}}\colon \mathcal{M}\rightarrow\mathbb{R}_+$ is defined by
\begin{equation}\label{MF}
\mathcal{I}_{\mathcal{E}}(u)=\int_\Omega \big(\nabla \frac{\delta}{\delta u}\mathcal{E}(u), A(u)G''(u)^{-1}\nabla \frac{\delta}{\delta u}\mathcal{E}(u)\big)dx.
\end{equation}
\end{proposition}
\begin{proof}
The proof follows directly from Definition \ref{def5}. In other words, 
\begin{equation*}
\begin{split}
\frac{d}{dt}\mathcal{E}(u)=&\int_\Omega \frac{\delta}{\delta u}\mathcal{E}(u)\cdot\partial_t u dx\\
=&\int_\Omega \Big[-\frac{\delta}{\delta u}\mathcal{E}(u) \nabla\cdot f_1(x,u)+\beta\frac{\delta}{\delta u} {\mathcal{E}(u)}\nabla\cdot(A(u)G''(u)^{-1}\nabla \frac{\delta}{\delta u}\mathcal{E}(u))\Big]dx\\
=&\int_\Omega (f_1(x,u), \nabla\frac{\delta}{\delta u}\mathcal{E}(u))dx-\beta\int_\Omega \big(\nabla \frac{\delta}{\delta u}\mathcal{E}(u), A(u)G''(u)^{-1}\nabla \frac{\delta}{\delta u}\mathcal{E}(u)\big)dx\\
=&-\beta\mathcal{I}_{\mathcal{E}}(u), 
\end{split}
\end{equation*}
where the third equality applies the integration by parts formula and the last equality uses condition \eqref{pgdcond}. 
\end{proof}
\begin{remark}
We remark that if $\mathcal{E}(u)=\mathcal{G}(u)=\int_\Omega (u\log u-u)dx$ and $A(u)=\mathbb{I}$, then the functional
\begin{equation*}
\mathcal{I}_{\mathcal{E}}(u)=\int_\Omega\|\nabla\log u\|^2 u dx=\int_\Omega \frac{\|\nabla u(x)\|^2}{u(x)}dx,
\end{equation*}
is known as the Fisher information functional. A known fact is that the dissipation of Lyapunov/entropy functional $\mathcal{G}(u)=\int_\Omega u\log u-u dx$ along the heat flow equals to the negative Fisher information functional. This fact is called the de Bruijn equality. Here, the de-Bruijn type equalities hold naturally in conservation laws with entropy-entropy flux pairs. In future works, we shall study the dynamical effect of flux function in metric spaces; see related techniques developed in \cite{LiG1}. 
\end{remark}
\begin{remark}
We remark that equation \eqref{pgd} are generalized variational formulations for flux--gradient flows in metric spaces. They have potential applications in designing Markov-Chain-Monte-Carlo algorithms, and deriving the neural network variational algorithms for conservation laws; see \cite{LOL}. We leave the detailed studies of these areas in future works. 
\end{remark}

\section{Controlling Conservation laws}\label{section4}
In this section, we present the main results of this paper. We study the variational problems for conservation laws \eqref{cld}. From now on, we assume that the entropy-entropy flux-metric condition in Definition \ref{def2} holds. 

We first design an optimal control problem over flux--gradient flows in a metric space. 
\begin{definition}[Optimal control of conservation laws]
\begin{subequations}\label{MFC}
Given smooth functionals $\mathcal{F}$, $\mathcal{H}\colon \mathcal{M}\rightarrow\mathbb{R}$, consider a variational problem
\begin{equation}\label{MFC1}
\inf_{u, v, u_1}~\int_0^1\Big[\int_\Omega \frac{1}{2}\big(v, A(u)G''(u)^{-1}v\big)dx-\mathcal{F}(u)\Big]dt+\mathcal{H}(u_1),
\end{equation}
where the infimum is taken among variables $v\colon [0,1]\times \Omega\rightarrow\mathbb{R}^n$, $u\colon [0,1]\times\Omega\rightarrow\mathbb{R}$, and $u_1\colon\Omega\rightarrow\mathbb{R}$ satisfying 
\begin{equation}\label{MFC2}
\begin{split}
\partial_t u+\nabla\cdot f(u)+\nabla\cdot(A(u)G''(u)^{-1}v)=\beta \nabla\cdot(A(u)\nabla u),\quad u(0,x)=u_0(x).
\end{split}
\end{equation}
\end{subequations}
\end{definition}
Here the equation \eqref{MFC2} is a control dynamic for the conservation law \eqref{cld}. The objective functional is an enforced ``kinetic energy'' minus a ``potential energy'' in the metric space $(\mathcal{M}, \g)$. If the control variable $v(t,x)=0$ for all $t=[0,1]$, $x\in \Omega$, then dynamics \eqref{MFC2} becomes the original conservation law equation \eqref{cld}. 
\begin{remark}
We notice that variational problem \eqref{MFC} is a generalized dynamical optimal transport problem. In other words, if $u\colon[0,1]\times\Omega\rightarrow\mathbb{R}_+$, $A(u)=\mathbb{I}$, $G(u)=u\log u-u$, $\beta=0$, $f=0$, $\mathcal{F}(u)=0$ and $u_1$ is a fixed function with $\int_\Omega u_0 dx=\int_\Omega u_1 dx$, then problem \eqref{MFC} forms 
  \begin{equation*}
\inf_{u, v}~\Big\{\int_0^1\int_\Omega \frac{1}{2}\|v\|^2udxdt\colon \partial_t u+\nabla\cdot(uv)=0,\quad u(0,x)=u_0(x),\quad u(1,x)=u_1(x)\Big\}.
\end{equation*}
The above minimization is known as Benamou-Brenier's formula \cite{BB} studied in classical optimal transport problems. Here the variational problem \eqref{MFC} generalizes the dynamical optimal transport, which contains the inverse of the Hessian operator of entropy functionals to model the ``kinetic energy''. In particular, we formulate the conservation laws in the constraint set.  
\end{remark}

We next derive critical point systems of variational problem \eqref{MFC}. They are Hamiltonian flows in $(\mathcal{M}, \g)$ associated with conservation laws.   
\begin{proposition}[Hamiltonian flows of conservation laws]\label{MFH}
The critical point system of variational problem \eqref{MFC} is given below. There exists a function $\Phi\colon [0,1]\times\Omega\rightarrow\mathbb{R}$, such that 
\begin{equation*}
v(t,x)=\nabla\Phi(t,x),
\end{equation*}
and
\begin{equation}\label{MFCeq}
\left\{\begin{aligned}
&\partial_t u+\nabla\cdot f(u)+\nabla\cdot(A(u)G''(u)^{-1}\nabla\Phi)=\beta \nabla\cdot(A(u)\nabla u),\\
&\partial_t\Phi+(\nabla\Phi, f'(u))+\frac{1}{2}(\nabla\Phi, (A(u)G''(u)^{-1})'\nabla\Phi)+\frac{\delta}{\delta u}\mathcal{F}(u)=-\beta \nabla\cdot(A(u)\nabla\Phi)+\beta(\nabla\Phi, A'(u)\nabla u).
\end{aligned}\right.
\end{equation}
Here $'$ represents the derivative w.r.t. variable $u$. The initial and terminal time conditions satisfy
\begin{equation*}
u(0,x)=u_0(x), \quad \frac{\delta}{\delta u_1}\mathcal{H}(u_1)+\Phi(1,x)=0.
\end{equation*}
\end{proposition}
\begin{proof}
We first rewrite the variables $(u, v)$ in variational formula \eqref{MFC} by $(u,m)$, where 
\begin{equation*}
m(t,x)=A(u(t,x))G''(u(t,x))^{-1}v(t,x)=V(u(t,x))v(t,x).
\end{equation*}
Here we denote $V(u(t,x))=A(u(t,x))G''(u(t,x))^{-1}$. In this case, the variational problem \eqref{MFC} forms
\begin{equation}\label{variation}
\begin{split}
&\inf_{m, u, u_1} \Big\{\int_0^1\Big[ \int_\Omega \frac{1}{2}(m, V(u)^{-1}m)dx-\mathcal{F}(u)\Big]dt+\mathcal{H}(u_1) \colon \\
&\qquad\quad\partial_t u+\nabla\cdot f(u)+ \nabla\cdot m=\beta \nabla\cdot(A(u)\nabla u), \quad \textrm{fixed $u_0$.}\Big\}.
\end{split}
\end{equation}
Denote the Lagrange multiplier of problem \eqref{variation} by $\Phi\colon [0,1]\times\Omega\rightarrow\mathbb{R}$. Consider the following saddle point problem 
\begin{equation*}
\begin{split}
\inf_{m,u,u_1}\sup_\Phi \quad \mathcal{L}(m, u, u_1, \Phi).
\end{split}
\end{equation*}
In the above formula, we have
\begin{equation*}
\begin{split}
\mathcal{L}(m,u,u_1,\Phi)= &\int_0^1 \int_\Omega \Big[\frac{1}{2}(m, V(u)^{-1}m)+\Phi\Big(\partial_t u+\nabla\cdot f(u)+ \nabla\cdot m-\beta \nabla\cdot(A(u)\nabla u)\Big)\Big] dxdt\\
&\quad-\int_0^1\mathcal{F}(u) dt+\mathcal{H}(u_1)\\
= &\int_0^1 \int_\Omega \Big[\frac{1}{2}(m, V(u)^{-1}m)+\Phi\Big(\nabla\cdot f(u)+ \nabla\cdot m-\beta \nabla\cdot(A(u)\nabla u)\Big)\Big] dxdt\\
&+\int_\Omega\big(\Phi(1,x)u_1(x)-\Phi(0,x)u_0(x)\big)dx-\int_0^1\int_\Omega \partial_t\Phi u dxdt-\int_0^1\mathcal{F}(u) dt+\mathcal{H}(u_1),
\end{split}
\end{equation*}
where we use the integration by parts formula w.r.t. $t$ in the second equality. In other words, 
\begin{equation*}
\int_0^1\int_\Omega\Phi\partial_tu dxdt=\int_\Omega \Phi(1,x)u(1,x) dx-\int_\Omega\Phi(0,x)u(0,x)dx-\int_0^1\int_\Omega \partial_t\Phi u dxdt. 
\end{equation*}
We next derive the critical point for the above saddle point problem. In other words, consider
\begin{equation*}
\left\{\begin{split}
&\frac{\delta}{\delta m}\mathcal{L}=0\\
 &\frac{\delta}{\delta u}\mathcal{L}=0\\
 &\frac{\delta}{\delta u_1}\mathcal{L}=0\\
 &\frac{\delta}{\delta \Phi}\mathcal{L}=0
\end{split}\right.\quad\Rightarrow\quad\left\{\begin{split}
&V(u)^{-1}m=\nabla\Phi,\\
 &-\frac{1}{2}(m, V(u)^{-1}V(u)'V(u)^{-1}m)-\frac{\delta}{\delta u}\mathcal{F}-\partial_t\Phi-(\nabla \Phi, f'(u))\\
 &\hspace{3.7cm}-\beta \nabla\cdot(A(u)\nabla\Phi)+\beta(\nabla \Phi, A'(u)\nabla u)=0,\\
 &\Phi_1+\frac{\delta}{\delta u_1}\mathcal{H}(u_1)=0,\\
 &\partial_t u+\nabla\cdot f(u)+ \nabla\cdot m-\beta \nabla\cdot(A(u)\nabla u)=0,
\end{split}\right.
\end{equation*}
where $\frac{\delta}{\delta m}$, $\frac{\delta}{\delta u}$, $\frac{\delta}{\delta u_1}$, $\frac{\delta}{\delta\Phi}$ are $L^2$ first variational derivatives w.r.t. functions $m$, $u$, $u_1$, $\Phi$, respectively. We thus derive the pair of PDEs \eqref{MFCeq} in $\mathcal{M}(\Omega)$.
\end{proof}
We next present the Hamiltonian formalism for the PDE system \eqref{MFCeq}. 
\begin{proposition}[Hamiltonian flows in metric space]\label{prop9}
PDE system \eqref{MFCeq} has the following Hamiltonian flow formulation. 
\begin{equation*}
\partial_tu=\frac{\delta}{\delta\Phi}\mathcal{H}_{\mathcal{G}}(u, \Phi),\quad \partial_t\Phi=-\frac{\delta}{\delta u}\mathcal{H}_{\mathcal{G}}(u, \Phi),
\end{equation*}
where we define the Hamiltonian functional $\mathcal{H}_{\mathcal{G}}\colon \mathcal{M}\times C^{\infty}(\Omega)\rightarrow\mathbb{R}$ by
\begin{equation}\label{eq:hamiltonian_functional}
\mathcal{H}_{\mathcal{G}}(u,\Phi)=\int_\Omega \Big[\frac{1}{2}(\nabla\Phi, A(u)G''(u)^{-1}\nabla\Phi)+(\nabla\Phi, f(u))-\beta (\nabla\Phi, A(u)\nabla u) \Big]dx+\mathcal{F}(u).
\end{equation}
In other words, the Hamiltonian functional $\mathcal{H}_{\mathcal{G}}(u, \Phi)$ is conserved along dynamics \eqref{MFCeq}. Suppose $(u, \Phi)$ are solutions of equation \eqref{MFCeq}, then
\begin{equation*}
\frac{d}{dt}\mathcal{H}_{\mathcal{G}}(u,\Phi)=0.
\end{equation*}
\end{proposition}
\begin{proposition}[Functional Hamilton-Jacobi equations of conservation laws]\label{prop10}
The Hamilton-Jacobi equation in $(\mathcal{M}, \g)$ for equation \eqref{MFCeq} satisfies
\begin{equation*}
\begin{split}
&\partial_t\mathcal{U}(t,u)+\int_\Omega\Big[ \frac{1}{2}\big(\nabla \frac{\delta}{\delta u(x)}\mathcal{U}(t,u), A(u)G''(u)^{-1}\nabla\frac{\delta}{\delta u(x)}\mathcal{U}(t,u)\big)+\big(\nabla\frac{\delta}{\delta u(x)}\mathcal{U}(t,u), f(u)\big)\\
&\hspace{6cm}-\beta (\nabla \frac{\delta}{\delta u(x)}\mathcal{U}(t,u), A(u)\nabla u)\Big]dx+\mathcal{F}(u)=0,
\end{split}
\end{equation*}
where $\mathcal{U}\colon [0,1]\times L^2(\Omega)\rightarrow\mathbb{R}$ is a value functional. 
\end{proposition}
\begin{proof}[Proof of Proposition \ref{prop9} and Proposition \ref{prop10}]
The proof follows from the definition of Hamiltonian dynamics. We can check it directly by using the first variation operators. In other words, 
\begin{equation*}
\frac{\delta}{\delta\Phi}\mathcal{H}_{\mathcal{G}}(u,\Phi)=-\nabla\cdot(A(u)G''(u)^{-1}\nabla\Phi)-\nabla\cdot f(u)+\beta\nabla\cdot(A(u)\nabla u),    
\end{equation*}
and 
\begin{equation*}
\begin{split}
\frac{\delta}{\delta u}\mathcal{H}_{\mathcal{G}}(u,\Phi)=&\frac{1}{2}(\nabla\Phi, (A(u)G''(u)^{-1})'\nabla\Phi)+(\nabla\Phi, f'(u))+\frac{\delta}{\delta u}\mathcal{F}(u)\\
&-\beta (\nabla\Phi, A'(u)\nabla u)+\beta\nabla\cdot(A(u)\nabla\Phi).    
\end{split}
\end{equation*}
Clearly, $\mathcal{H}_{\mathcal{G}}(u,\Phi)$ is conserved since 
\begin{equation*}
\begin{split}
\frac{d}{dt}\mathcal{H}_{\mathcal{G}}(u,\Phi)=&\int_\Omega \Big(\frac{\delta}{\delta u}\mathcal{H}_{\mathcal{G}}\cdot \partial_t u+\frac{\delta}{\delta \Phi}\mathcal{H}_{\mathcal{G}}\cdot\partial_t\Phi\Big)dx \\
=&\int_\Omega \Big(\frac{\delta}{\delta u}\mathcal{H}_{\mathcal{G}}\cdot \frac{\delta}{\delta\Phi}\mathcal{H}_{\mathcal{G}}-\frac{\delta}{\delta \Phi}\mathcal{H}_{\mathcal{G}}\cdot\frac{\delta}{\delta u}\mathcal{H}_{\mathcal{G}}\Big) dx \\
=&0. 
\end{split}
\end{equation*}
This finishes the proof of Proposition \ref{prop9}.

We next use the Hamiltonian functional to formulate the Hamilton-Jacobi equation in $(\mathcal{M}, \g)$. Define a functional by $\mathcal{U}\colon \mathbb{R}_+\times\mathcal{M}\rightarrow\mathbb{R}$. Let $\Phi(t,x)=\frac{\delta}{\delta u(x)}\mathcal{U}(t, u)$ in Hamiltonian flow \eqref{MFCeq}. We obtain the Hamilton-Jacobi equation in metric space $(\mathcal{M}, \g)$, where
\begin{equation*}
\partial_t\mathcal{U}(t,u)+\mathcal{H}_{\mathcal{G}}(u, \frac{\delta}{\delta u}\mathcal{U}(t,u))=0. 
\end{equation*}
This finishes the proof of Proposition \ref{prop10}. 
\end{proof}
\begin{remark}[Fisher information regularizations]
We remark that the Fisher information functional is also connected with control problems of conservation laws. See similar studies in \cite{LL}. We leave the study about information functional regularizations of conservation laws in a sequential work.  
\end{remark}
\subsection{Examples}\label{section6}
In this subsection, we list several examples of control problems for scalar conservation laws. 
\begin{example}[Controlling heat equations]
Consider the heat equation
\begin{equation*}
\partial_t u=\beta\Delta u,
\end{equation*}
where $u\colon [0,+\infty)\times \Omega\rightarrow\mathbb{R}_+$ is the probability density function.
It satisfies PDE \eqref{cld}, where
\begin{equation*}
u\in\mathbb{R}_+^1, \quad f(u)=0,\quad A(u)=\mathbb{I}.
\end{equation*}
In this case, a function pair $(G, \Psi)$, where $G$ is a convex function and $\Psi=0$, satisfies the entropy-entropy flux condition. In other words,
\begin{equation*}
\Psi'(u)=f'(u)G'(u)=0.
\end{equation*}
In particular, if $G(u)=u\log u-u$, then the variational problem \eqref{MFC} satisfies 
\begin{equation*}
\inf_{u, v, u_1}~\int_0^1\Big[\int_\Omega \frac{1}{2}|v|^2udx-\mathcal{F}(u)\Big]dt+\mathcal{H}(u_1),
\end{equation*}
where the infimum is taken among variables $u, v, u_1$ satisfying 
\begin{equation*}
\partial_t u+\nabla\cdot(uv)=\beta \Delta u,\quad u(0,x)=u_0(x).
\end{equation*}
Here the minimizer system is given below. There exists a function $\Phi$, such that $v=\nabla\Phi$ and
\begin{equation*}
\left\{\begin{aligned}
&\partial_t u+\nabla\cdot(u\nabla\Phi)=\beta \Delta u,\\
&\partial_t\Phi+\frac{1}{2}\|\nabla\Phi\|^2+\frac{\delta}{\delta u}\mathcal{F}(u)=-\beta \Delta \Phi.
\end{aligned}\right.
\end{equation*}
The above optimal control problem and critical point system has been widely studied in the optimal transport ($\mathcal{F}=0$, $\beta=0$), Schr{\"o}dinger bridge problems ($\mathcal{F}=0$, $\beta> 0$) and potential mean field games. There is a Hamiltonian formalism for the PDE system $(u,\Phi)$. In other words, \begin{equation*}
\partial_t u=\frac{\delta}{\delta\Phi}\mathcal{H}_{\mathcal{G}}(u,\Phi), \quad \partial_t \Phi=-\frac{\delta}{\delta u}\mathcal{H}_{\mathcal{G}}(u,\Phi),
\end{equation*}
where the Hamiltonian functional satisfies 
\begin{equation*}
\mathcal{H}_{\mathcal{G}}(u,\Phi)=\int_\Omega \Big[\frac{1}{2}\|\nabla\Phi\|^2 u-\beta (\nabla \Phi, \nabla u)\Big] dx+\mathcal{F}(u).
\end{equation*}
\end{example}
\begin{example}[Controlling scalar conservation laws]
Consider
\begin{equation*}
\partial_tu+\nabla\cdot f(u)=\beta\Delta u,
\end{equation*}
where $u\colon [0,+\infty)\times \Omega\rightarrow\mathbb{R}$ and $A=\mathbb{I}$. 
\begin{itemize}
\item[(i)] Let $G(u)=u\log u-u$. Then variational problem \eqref{MFC} satisfies
\begin{equation*}
\inf_{u, v, u_1}~\int_0^1\Big[\int_\Omega \frac{1}{2}|v|^2udx-\mathcal{F}(u)\Big]dt+\mathcal{H}(u_1),
\end{equation*}
where the infimum is taken among variables $u, v, u_1$ satisfying 
\begin{equation*}
\partial_t u+\nabla\cdot f(u)+\nabla\cdot(uv)=\beta \Delta u,\quad u(0,x)=u_0(x).
\end{equation*}
Here the minimizer system is given below. There exists a function $\Phi$, such that $v=\nabla\Phi$ and
\begin{equation*}
\left\{\begin{aligned}
&\partial_t u+\nabla\cdot f(u)+\nabla\cdot(u\nabla\Phi)=\beta \Delta u,\\
&\partial_t\Phi+(\nabla\Phi, f'(u))+\frac{1}{2}\|\nabla\Phi\|^2+\frac{\delta}{\delta u}\mathcal{F}(u)=-\beta \Delta \Phi.
\end{aligned}\right.
\end{equation*}
I.e., 
\begin{equation*}
\partial_t u=\frac{\delta}{\delta\Phi}\mathcal{H}_{\mathcal{G}}(u,\Phi), \quad \partial_t \Phi=-\frac{\delta}{\delta u}\mathcal{H}_{\mathcal{G}}(u,\Phi),
\end{equation*}
where the Hamiltonian functional satisfies 
\begin{equation*}
\mathcal{H}_{\mathcal{G}}(u,\Phi)=\int_\Omega \Big[\frac{1}{2}\|\nabla\Phi\|^2 u +(\nabla\Phi, f(u))-\beta (\nabla \Phi, \nabla u)\Big] dx+\mathcal{F}(u).
\end{equation*}

\item[(ii)]  Let $G(u)=\frac{u^2}{2}$. Then variational problem \eqref{MFC} satisfies
\begin{equation*}
\inf_{u, v, u_1}~\int_0^1\Big[\int_\Omega \frac{1}{2}|v|^2dx-\mathcal{F}(u)\Big]dt+\mathcal{H}(u_1),
\end{equation*}
where the infimum is taken among variables $u, v, u_1$ satisfying 
\begin{equation*}
\partial_t u+\nabla\cdot f(u)+\nabla\cdot v=\beta \Delta u,\quad u(0,x)=u_0(x).
\end{equation*}
Here the minimizer system is given below. There exists a function $\Phi$, such that $v=\nabla\Phi$ and
\begin{equation*}
\left\{\begin{aligned}
&\partial_t u+\nabla\cdot f(u)+\nabla\cdot(\nabla\Phi)=\beta \Delta u,\\
&\partial_t\Phi+(\nabla\Phi, f'(u))+\frac{\delta}{\delta u}\mathcal{F}(u)=-\beta \Delta \Phi.
\end{aligned}\right.
\end{equation*}
I.e.,
\begin{equation*}
\partial_t u=\frac{\delta}{\delta\Phi}\mathcal{H}_{\mathcal{G}}(u,\Phi), \quad \partial_t \Phi=-\frac{\delta}{\delta u}\mathcal{H}_{\mathcal{G}}(u,\Phi),
\end{equation*}
where the Hamiltonian functional satisfies 
\begin{equation*}
\mathcal{H}_{\mathcal{G}}(u,\Phi)=\int_\Omega \Big[\frac{1}{2}\|\nabla\Phi\|^2 +(\nabla\Phi, f(u))-\beta (\nabla \Phi, \nabla u)\Big] dx+\mathcal{F}(u).
\end{equation*}
\end{itemize}
\end{example}
	\section{Numerical methods and examples}\label{primal-dual}
In this section, we first review classical primal-dual hybrid gradient algorithms (PDHG) and their extensions. We next apply this approach to solve the variational problem defined in equation \eqref{MFC} subject to the constraint involving conservation laws. We design a finite difference scheme to discretize conservation laws and solve the variational problem on grids. Several numerical examples, including Burgers' equation and traffic flow, are presented. 
	\subsection{PDHG algorithm and its extension}
The PDHG algorithm \cite{champock11} solves the following constrained convex optimization problem 
\begin{align*}
\min_{z}  h(Kz) + g(z),
\end{align*}
where $\mathcal{Z}$ is a finite or infinite dimensional Hilbert space, $h$ and $g$ are convex functions and $K:\mathcal{Z}\rightarrow \mathcal{H}$ is a linear operator between Hilbert spaces. This problem can be rewritten in the saddle-point problem form
\begin{align*}
\min_z \max_p \langle Kz,p\rangle_{L^2} + g(z) -h^*(p),
\end{align*}
where $h^*(p) = \sup_z \langle Kz,p\rangle_{L^2}  -h(z)$, which is the convex conjugate of $h$. The algorithm updates $z$, $p$ by taking proximal gradient descent, proximal gradient ascent steps alternatively. At the $n$-th iteration, the algorithm updates as follows
\begin{align*}
z^{n+1} &= \argmin_{z}   \langle Kz,\bar{p}^n\rangle_{L^2}  + g(z) +  \frac{1}{2 \tau} \| z - z^n\|^2_{L^2},\\
p^{n+1} &= \argmax_{p}   \langle Kz^{n+1},p\rangle_{L^2}  -h^*(p) -\frac{1}{2 \sigma} \| p - p^n\|^2_{L^2},\\
\bar{p}^{n+1} &= 2p^{n+1} - p^{n}.
\end{align*}
Here $\tau$, $\sigma$ are stepsizes, which have to satisfy $\sigma \tau \|K^T K\|<1 $ in order to guarantee the convergence of the algorithm. When the operator $K$ is nonlinear, we use the extension of PDHG algorithm \cite{clason2017primal}. The idea is to use the linear approximation of $K$:
\begin{align*}
K(z) \approx K(\bar{z}) + \nabla K(\bar{z})(z - \bar{z}).
\end{align*}
Here, the extension of the PDHG scheme is as follows 
\begin{equation}\label{alg:pdhg_nonlinear}
\begin{aligned}
z^{n+1} & = \argmin_{z}   \langle z,[\nabla K(z^n)]^T\bar{p}^n\rangle_{L^2}  + g(x) +  \frac{1}{2 \tau} \| z -  z^n\|^2_{L^2},\\
p^{n+1} & = \argmax_{p}   \langle K(z^{n+1}),p\rangle_{L^2}  -h^*(p) - \frac{1}{2 \sigma} \| p - p^n\|^2_{L^2},\\
\bar{p}^{n+1} &= 2p^{n+1} - p^{n}.
\end{aligned}
\end{equation}	
	When $K$ is some unbounded linear operator, for instance $K = \nabla$, the operator norm $\|K\|$ can increase when we refine the grid size. Consequently, the algorithm may converge slowly due to small stepsizes. We apply a generalization of PDHG, namely the General-proximal Primal-Dual Hybrid Gradient (G-prox PDHG) method from \cite{JacobsLegerLiOsher2018_solvinga}. We choose a proper norm ($L^2, H^1, ...$) for the proximal step to allow larger stepsizes.
	
	We apply the nonlinear G-prox PDHG algorithm to solve the variational problem \eqref{MFC}, in particular its equivalent format \eqref{variation}. For illustration purpose, we use the traffic flow variational problem \eqref{main_variation_example_2} as an example and give details on the algorithm. 
	Set
	\begin{align*}
z & = (u,m),\\
p & = \Phi,\\
K\left((u,m)\right) & = \partial_tu+ \nabla \cdot f(u) +\nabla \cdot m - \beta \Delta u , \\
g\left((u,m)\right) & = \int_0^1 \left(\int_{\Omega} \frac{\|m\|^2}{2u} +\mathbf{1}_{[0,1]}(u)  dx - \mathcal{F}(u) \right)dt,\\
h(Kz) & = \begin{cases}
0 \quad \text{if} \;Kz = 0\\
+\infty \quad \text{else}
\end{cases}.
\end{align*}
	We rewrite the varational problem as follows:
	\begin{equation}\label{eqn:minmax0}
	\inf_{u,m} \sup_{\Phi}~~\mathcal{L}(u,m,\Phi),\quad \text{with}\;u(0,x) = u_0(x),~\Phi(1,x)=-\frac{\delta}{\delta u(1,x)} \mathcal{H}(u),
	\end{equation}
	where
	\begin{equation}\label{eq:L_pdhg_cont}
	\begin{aligned}{}
	\mathcal{L}(u,m,\Phi) = &\int_0^1 \left(\int_{\Omega} \frac{\|m\|^2}{2 u} +\mathbf{1}_{[0,1]}(u)  dx - \mathcal{F}(u) \right)dt \\
	&+ \int_0^1\int_{\Omega} \Phi \left( \partial_tu + \nabla \cdot f(u) +\nabla \cdot m - \beta \Delta u  \right)  dx dt,
		\end{aligned}
	\end{equation}
and $f(u) = u(1-u)$ is the traffic flux function. We denote the indicator function by
	\begin{equation*}
\mathbf{1}_{A} (x) =
 \begin{cases}
		+\infty \quad    x\notin A\\
		0 \quad \quad \; x \in A.
		\end{cases}
	\end{equation*}
We also choose $L^2$ norm for updating $(u,m)$ and $H^1$ norm for $\Phi$, specifically
\begin{align*}
    \|v\|^2_{L^2} =\int_0^1 \int_{\Omega}v^2 dxdt,\quad \|v\|^2_{H_1^2} =  \|\nabla v\|^2_{L^2} + \|\partial_t v\|^2_{L^2}. 
\end{align*}

We summarize the algorithm to solve problem \eqref{eqn:minmax0} as follows.
	\begin{tabbing}
		aaaaaaaaa\= aaa \=aaa\=aaa\=aaa\=aaa=aaa\kill  
		\rule{\linewidth}{0.8pt}\\
		\noindent{\large\bf Algorithm 1:PDHG for the conservation law control system}\\

		\rule{\linewidth}{0.5pt}\\
		\+\+\+\textbf{While} $k <$ Maximal number of iteration\\ 
		[2ex]   
		\' $\left(u^{(k+1)},m^{(k+1)}\right)$ \'  $= \text{argmin}_u~~ \mathcal{L}(u, m,\bar{\Phi}^{(k)}) + \frac{1}{2\tau} \|u - u^{(k)}\|^2_{L^2} + \frac{1}{2\tau} \|m - m^{(k)}\|^2_{L^2}$;\\
		[2ex]   
		\' $\Phi^{(k+ 1)}$ \' $= \text{argmax}_\Phi~~\mathcal{L}(u^{(k+1)}, m^{(k+1)},\Phi) - \frac{1}{2\sigma} \|\Phi - \Phi^{(k)}\|^2_{H_1^2}$;\\
		[1ex]   
		\-\-\-   \' $\bar{\Phi}^{(k+1)}$ \' $= 2 \Phi^{(k+1)} - \Phi^{(k)}$;\\
		\rule{\linewidth}{0.5pt}
	\end{tabbing}

	\subsection{Finite Difference Discretization}
In this subsection, we review basic numerical concepts of conservation law. More details can be found in \cite{leveque1992numerical}. Then we design the finite difference discretization for the control of conservation laws. 
	\subsubsection{Lax--Friedrichs scheme for the conservation law}
Consider a nonlinear scalar conservation law
\begin{equation}
\begin{cases}
	\partial_t u + \partial_xf(u) = \beta \partial_{xx}u,\\
	u(0,x) = u_0(x).
\end{cases}
\end{equation}
We review the Lax--Friedrichs scheme. Denote a discretization ${\Delta t},{\Delta x}$ in time and space, $u_j^k = u(k \Delta t, j\Delta x)$, then the update follows
\begin{equation}
\label{eq:lax_friedrichs}
	u_j^{k+1} = u_j^{k} - \frac{\Delta t}{2\Delta x}\left( f(u_{j+1}^k) -f(u_{j-1}^k) \right) + (\beta  + c \Delta x ) \frac{\Delta t }{(\Delta x)^2} \left(u_{j+1}^k -2 u_j^k + u_{j-1}^k \right).
\end{equation}
The last term contains diffusion coefficient $\beta$ from the original equation and $c\Delta x$ as coefficient for artificial viscosity with $c>0$.
Notice that the monotone scheme gives entropy solutions. Here the definition of monotone scheme is given below. 
\begin{definition}
	For $p,q \in\mathbb{N}$, a scheme 
	\begin{align*}
	u_j^{k+1} = G(u_{j-p-1}^k, ..., u_{j+q}^k)
	\end{align*}
	is called a \text{monotone scheme} if $G$ is a monotonically nondecreasing function of each argument.
\end{definition}
In order to guarantee that the scheme \eqref{eq:lax_friedrichs} is monotone, the following inequalities have to be satisfied:
\begin{align*}
	1-2(\beta + c \Delta x )\frac{\Delta t }{(\Delta x)^2} \geq 0,\\
	-\frac{\Delta t}{2\Delta x}|f'(u)| + (\beta + c \Delta x )\frac{\Delta t }{(\Delta x)^2}  \geq 0.
\end{align*}
As we want the scheme works when $\beta \rightarrow 0$, the restriction on $c$ and space--time stepsizes can be simplified as follows:
\begin{align*}
	c\geq \frac{1}{2}|f'(u)|,\\
	 {(\Delta x)^2} \geq 2(\beta + c \Delta x) \Delta t. 
\end{align*} 
The first inequality suggests the artificial viscosity we need to add. The second one impose a strong restriction on the stepsize in time when $\beta >0$.
	\subsubsection{Discreitization of the control problem}
	We consider the control problem of scalar conservation law defined in $[0,b]\times [0,1]$, where $b$ is a given constant. We apply the periodic boundary condition on the spatial domain. Given $N_{x},N_t>0$, we have $\Delta x = \frac{b}{N_{x}}$, $\Delta t = \frac{1}{N_{t}}$.
	For $x_i = i \Delta x,t_l =l \Delta t $, define  
	\begin{align*}
	& u_{i}^l = u(t_l,x_i) &\quad 1\leq i \leq N_{x}, 0\leq l \leq N_t,\\
	& m_{1,i}^l = \left(m_{x_1}(t_l,x_i)\right)^+  &\quad 1\leq i \leq N_{x},  0\leq l \leq  N_t-1,\\
	& m_{2,i}^l = -\left(m_{x_1}(t_l,x_i)\right)^-  &\quad 1\leq i \leq N_{x},  0\leq l \leq N_t-1,\\
	& \Phi_{i}^l = \Phi(t_l,x_i)  &\quad 1\leq i \leq N_{x}, 0\leq l \leq N_t,\\
	& \Phi_{i}^{Nt} = -\frac{\delta}{\delta u(1,x_i)} \mathcal{H}(u_i^{N_t}) &\quad 1\leq i \leq N_{x},
	\end{align*}
	where $u^+ := \max(u,0)$ and $u^- = u^+ -u$. Note here $m_{1,i}^l \in \mathbb{R}_+, m_{2,i}^l \in \mathbb{R}_{-} $. Denote
	\begin{align*}
	&	\left(D u\right)_{i} := \frac{u_{i+1} - u_{i}}{\Delta x}\\
	&	[D u]_{i} :=\left( \left(D u\right)_{i},\left(D u\right)_{i-1}\right)\\
	&	\widehat{[D u]}_{i}  = \left( \left(D u\right)^+_{i},-\left(D u\right)^-_{i-1}\right)\\
	& Lap(u)_{i} = \frac{u_{i+1} - 2u_{i} + u_{i-1}}{(\Delta x)^2}
	\end{align*}
	The first conservation law equation adapted from the Lax--Friedrichs scheme is as follows:
	\begin{equation}
	\frac{1}{\Delta t} \left( u_{i}^{l+1} - u_{i}^{l}\right) + 	\frac{1}{2\Delta x} \left( f(u_{i+1}^{l+1}) - f(u_{i-1}^{l+1})\right) + \left(D m\right)^{l}_{1,i-1}  +\left(D m\right)^{l}_{2,i}  = (\beta + c\Delta x) Lap(u)^{l+1}_i, 
	\end{equation}
	where $1\leq i \leq N_{x}$,  $0\leq l \leq N_t-1$. We choose $c \geq \frac{1}{2}\max_u |f'(u)|$. Unlike the Lax--Friedrichs scheme in explicit form, we use an implicit discretization in time to encode the feedback control structure. Another benefit of the implicit scheme is that it allows a larger mesh size in time. 
	
	Following the discretization of the conservation law, the discrete saddle point problem has the following form:
\begin{align*}
\min_{u,m} \max_{\Phi} L(u,m,\Phi),
\end{align*}
where
\begin{equation*}
\begin{aligned}
L(u,m,\Phi)= &\Delta x \Delta t \sum_{\substack{1 \leq i \leq N_x\\1 \leq l \leq N_t}} \frac{(m_{1,i}^{l-1})^2 +(m_{2,i}^{l-1})^2}{2 u_i^l}-\Delta t \sum_{ 1 \leq l \leq N_t} \mathcal{F}(u^l)\\
&+ \Delta x \sum_{1 \leq i \leq N_x}\mathcal{H}( u_{i}^{N_t}) + \sum_{\substack{1 \leq i \leq N_x\\1 \leq l \leq N_t}} \mathbf{1}_{[0,1]}(u_i^l)\\
& + \Delta x \Delta t \sum_{\substack{1 \leq i \leq N_x\\ 0 \leq l \leq N_t-1}} \Phi_{i}^{l} \Big( \frac{1}{\Delta t} ( u_{i}^{l+1} - u_{i}^{l}) + 	\frac{1}{2\Delta x} ( f(u_{i+1}^{l+1}) - f(u_{i-1}^{l+1}))\\
&\hspace{2.5cm}+(D m)^{l}_{1,i-1}  +(D m)^{l}_{2,i} - (\beta + c\Delta x) Lap(u)^{l+1}_i \Big).
\end{aligned}
\end{equation*}

By taking the first order derivative of $u_i^l$, we automatically get the implicit finite difference scheme for the dual equation of $\Phi$  that is backward in time.
The positive and negative parts of $(D\phi)_i^l$ are split, which help enhance the monotonicity of the discrete Hamiltonian.

	\begin{equation}
		\frac{1}{\Delta t} \left( \Phi_i^{l+1} - \Phi_i^{l}\right) + \frac{(\Phi_{i+1}^l - \Phi_{i-1}^l)}{2\Delta x} \left( f'(u_{i}^{l})\right) + \frac{1}{2} \|[\widehat{D \Phi}]_{i}^l \|^2 + \frac{\delta \mathcal{F}(u_i^l)}{\delta u} = -(\beta + c\Delta x) Lap(\Phi)^{l}_i,
	\end{equation}
			 for $1\leq i \leq N_{x}$, $1 \leq l \leq N_t$.
	   
We remark that the above discretizations, when $f(u) = 0$, reduce to the finite difference scheme for the mean-field game system proposed in \cite{briceno2019implementation}. The discrete form of the Hamiltonian functional \eqref{eq:hamiltonian_functional} at $t = t_l$ takes the form
\begin{equation*}
\begin{aligned}
   H_{\mathcal{G}}(u,\Phi) =\sum_{\substack 1\leq i \leq N_x} \left(\frac{1}{2} \|[\widehat{D \Phi}]_{i}^l \|^2u_i^l + \frac{(\Phi_{i+1}^l - \Phi_{i-1}^l)}{2\Delta x} \left( f(u_{i}^{l})\right) + (\beta + c\Delta x)u_i^l Lap(\Phi)^{l}_i \right) - \mathcal{F}(u^l).
\end{aligned}
\end{equation*}
We shall verify the conservation of Hamiltonian functional with numerical examples.

\subsubsection{Solve conservation laws via primal-dual algorithms}
When $\mathcal{F} =0$, $\mathcal{H} = c$ for some constant $c$, the variational problem \eqref{eqn:minmax0} becomes classical conservation laws with initial data. In this case, no control will be enforced on the density function $u$.
Therefore, the density movement is only determined by the flux term $\left(f(u)\right)_x.$
This means that the problem is reduced to initial value conservation laws. In this scenario, our approach proposed in Algorithm 1 provides an alternative way to solve the nonlinear conservation law with implicit discretization in time. In the implementation, we observed that this method allows a larger mesh size in time, thanks to the primal--dual variational structure. The method is also highly parallelizable as the proximal gradient descent step for $(u,m)$ is point-wise operation for each $(l,i)$. We demonstrate this part with two examples in the next section.
	
\subsection{Numerical examples}
In this subsection, we present numerical examples for control of conservation laws in one dimensional space. 
		\subsection{Example 1. }
	We consider the Burgers' equation on $(t,x) \in [0,1] \times [0,4]$.
	\begin{align*}
	\partial_t u + \partial_x f(u) = 0,\quad u(0,x) = \begin{cases}
	1 \quad &2\leq x \leq 3\\
	0 \quad \quad &\text{else}
	\end{cases},
	\end{align*}
	where $f(u) = \frac{1}{2}u^2, \mathcal{F} = 0, \mathcal{H} = 0, \beta = 0, k = 0.5$. The entropy solution to this problem at $t=1$ satisfies
	\begin{equation*}
		u(1,x) = \begin{cases}
		1 \quad & 3 \leq x \leq 3.5 \\
		(x-2) \quad &2 < x \leq 3\\
		0 \quad \quad &\text{else}
		\end{cases}.
	\end{equation*}
	In the following examples, we use the same spatial time domain with $N_t = 50, N_x = 100$ for the finite difference scheme. We solve the Burgers' equation using two approaches: solve the control problem; use the forward Lax--Friedrichs scheme. Figure \ref{fig:burgers_eg1} shows the solutions to the Burgers' equation. We can see that both numerical solutions are consistent with the exact entropy solution despite some numerical diffusions. From two plots on the right, there are clear formations of rarefaction wave and shock. We have verified that the shock travels at speed $v = \frac{1}{2}$.
	\begin{figure}[htbp!]
		\includegraphics[width=1\textwidth,trim=90 20 90 50, clip=true]{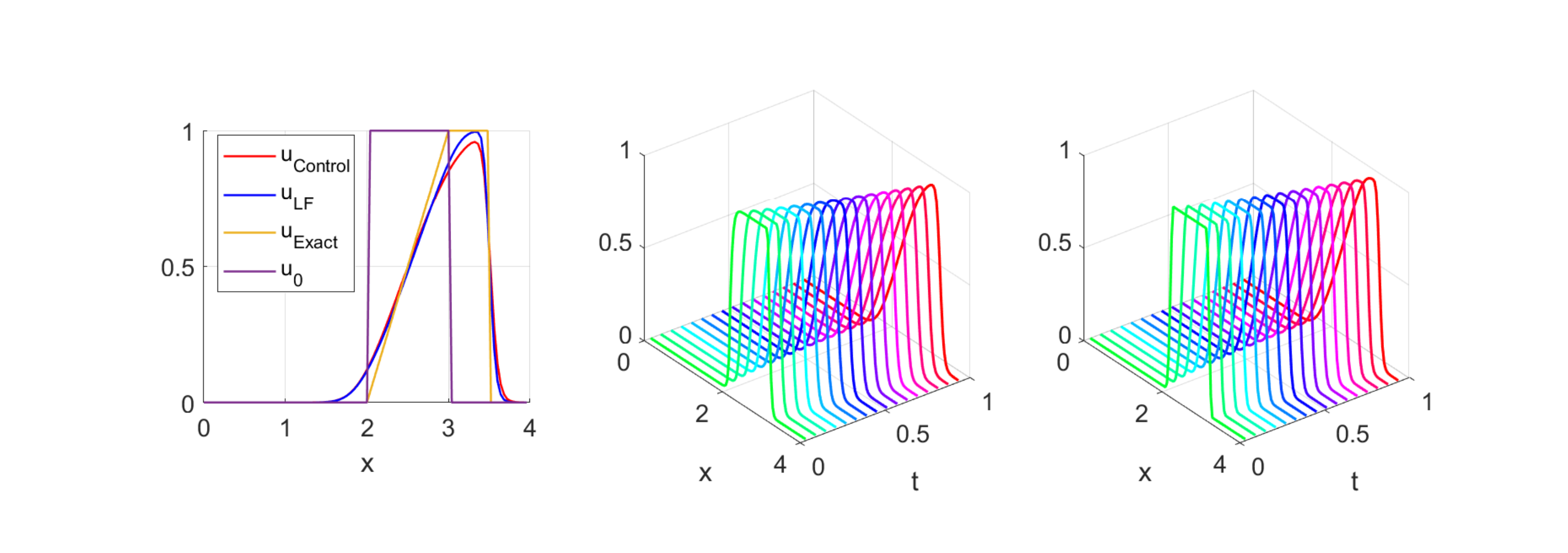}
		\caption{Numerical results for the Burgers' equation. Left:  a comparison with the exact solution at $t= 1$; middle: the numerical solution via solving the control problem; right: the numerical solution to the conservation law using Lax--Friedrichs scheme. }
		\label{fig:burgers_eg1}
	\end{figure}
	
	\subsection{Example 2. }
	We consider the traffic flow equation 
	\begin{align*}
		\partial_t u + \partial_x f(u) = 0,\quad u(0,x) = \begin{cases}
		0.8, \quad &1\leq x \leq 2\\
		0 \quad \quad &\text{else}
		\end{cases},
	\end{align*}
	where $f(u) = \frac{1}{2}u(1-u), \mathcal{F} = 0, \mathcal{H} = 0, \beta = 0, k = 0.5$.
\begin{figure}[htbp!]
	\includegraphics[width=1\textwidth,trim=75 20 90 50, clip=true]{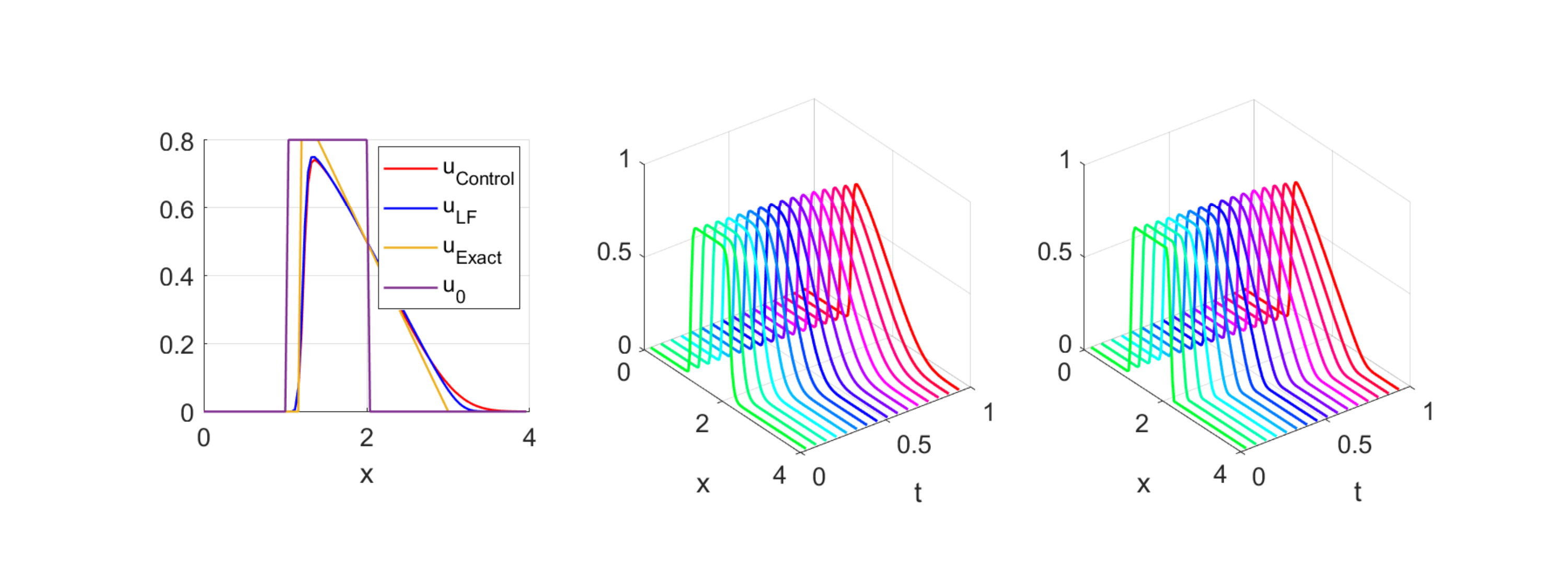}
	\caption{Numerical results for the traffic flow equation. Left:  a comparison with the exact solution at $t= 1$; middle: the numerical solution via solving the control problem; right: the numerical solution to the conservation law using Lax--Friedrichs scheme. }
	\label{fig:traffic_eg2}
\end{figure}
The entropy solution to this problem at time $t=1$ is
\begin{equation*}
	u(1,x) = \begin{cases}
	1 \quad & 1.2 \leq x \leq 1.4 \\
	\frac{1}{2}(3-x) \quad &1.4 < x \leq 3\\
	0 \quad \quad &\text{else}
	\end{cases}.
\end{equation*} 
We can see from Figure \ref{fig:traffic_eg2} (left) that the solution to the control problem gives the same solution to the conservation law, which matches the entropy solution. In the two plots on the right, we observe that both shocks and rarefaction waves are formed. This traffic flow example describes a group of cars $u_0$ waiting at the traffic light at $x=2$. At the time $t=0$, the red light turns green. But this group of cars doesn't move at a uniform constant speed. Instead, the car, whose originally position at time $t = 0$ is closer to the red light ($x=0$), moves faster. The density at position $x<0$  doesn't change until time $t = \frac{1}{2}(3-x)$.

\subsection{Example 3}
	We again consider the traffic flow equation with $f(u) = u(1-u), \beta= 0.1, \mathcal{F} = 0, c=0.5$. The final cost functional $ \mathcal{H}(u(1,\cdot)) = \mu \int_{\Omega}u(1,x) \log(\frac{u(1,x)}{u_1})dx, \mu >0$. We set $\mu =1$. In this case, the density $u(1,\cdot)$ will tend to form a `similar' distribution as $u_1$. When $\int_{\Omega}u_1 dx = \int_{\Omega}u_0 dx$ and $\mu \rightarrow + \infty$, this final cost functional is equivalent to imposes the constraint that $u(1,\cdot) = u_1$.  
We also compare the result from the control of conservation law with a mean-field game problem, i.e., $f=0$:
\begin{equation} \label{eq:mfg}
\begin{cases}
\partial_tu +\nabla \cdot (u \nabla \Phi)= \beta \Delta u,\\
\partial_t\Phi  +\frac{1}{2} \|\nabla \Phi\|^2 + \frac{\delta }{\delta u}\mathcal{F}(u)= -\beta \Delta \Phi,\\
u(0,x)=u_0(x),~\Phi(1,x)=-\frac{\delta}{\delta u_1} \mathcal{H}(u(1,\cdot)).
\end{cases}
\end{equation}
As shown in Figure \ref{fig:eg2_bdry}, we set
\begin{align*}
    u_0 &= 0.001 + 0.9e^{-10(x-2)^2},\\
    u_1 &= 0.001 + 0.45e^{-10(x-1)^2} + 0.45e^{-10(x-3)^2}.
\end{align*}
The solutions are presented in Figure \ref{fig:eg2_comparison}. We can see that in the mean-field game setup, there is an even split of the density. While for the control of conservation law, more density travels towards the location of the right-side Gaussian distribution. This demonstrates the difference between solutions in mean-field games and the ones in control of conservation laws. 
\begin{figure}[htbp!]
	\includegraphics[width=1.0\textwidth,trim=30 30 30 30, clip=true]{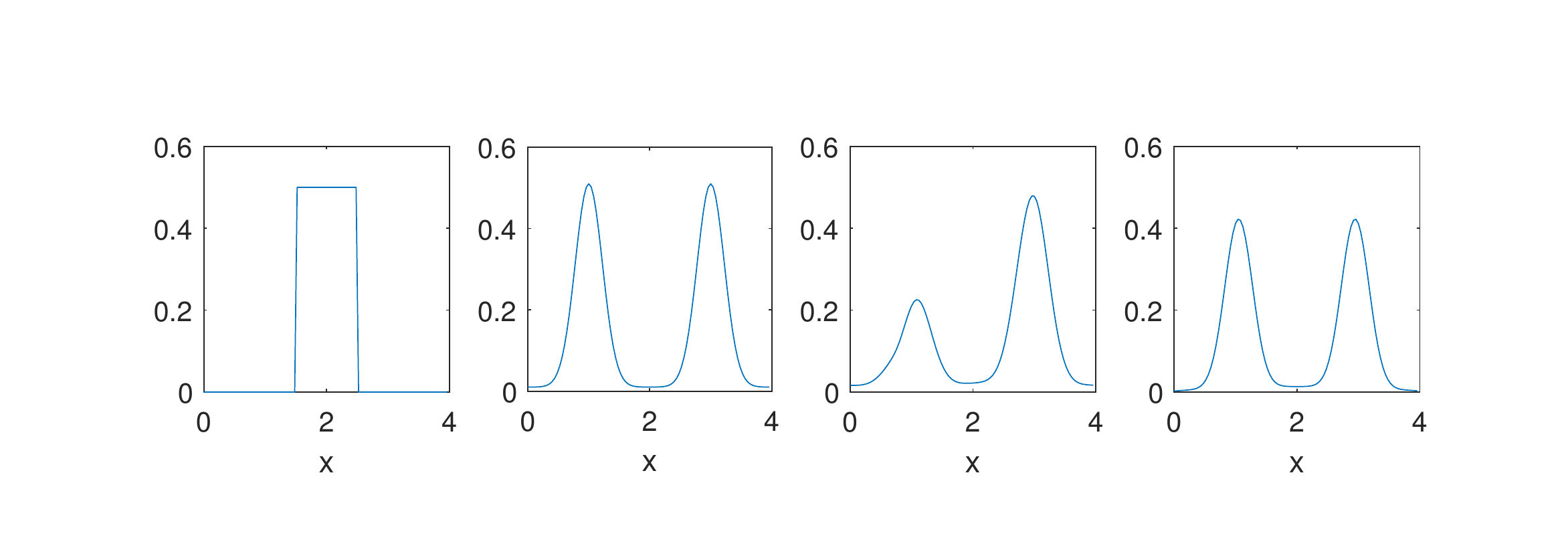}
	\caption{From left to right: initial configurations of $u_0 = 0.001 + 0.9e^{-10(x-2)^2}$, $u_1 = 0.001 + 0.45e^{-10(x-1)^2} + 0.45e^{-10(x-3)^2}$, solution $u(1,x)$ for the control of conservation law, solution $u(1,x)$ for the mean-field game problem.}
	\label{fig:eg2_bdry}	
\end{figure}

\begin{figure}[htbp!]
	\includegraphics[width=0.990\textwidth,trim=20 5 20 10, clip=true]{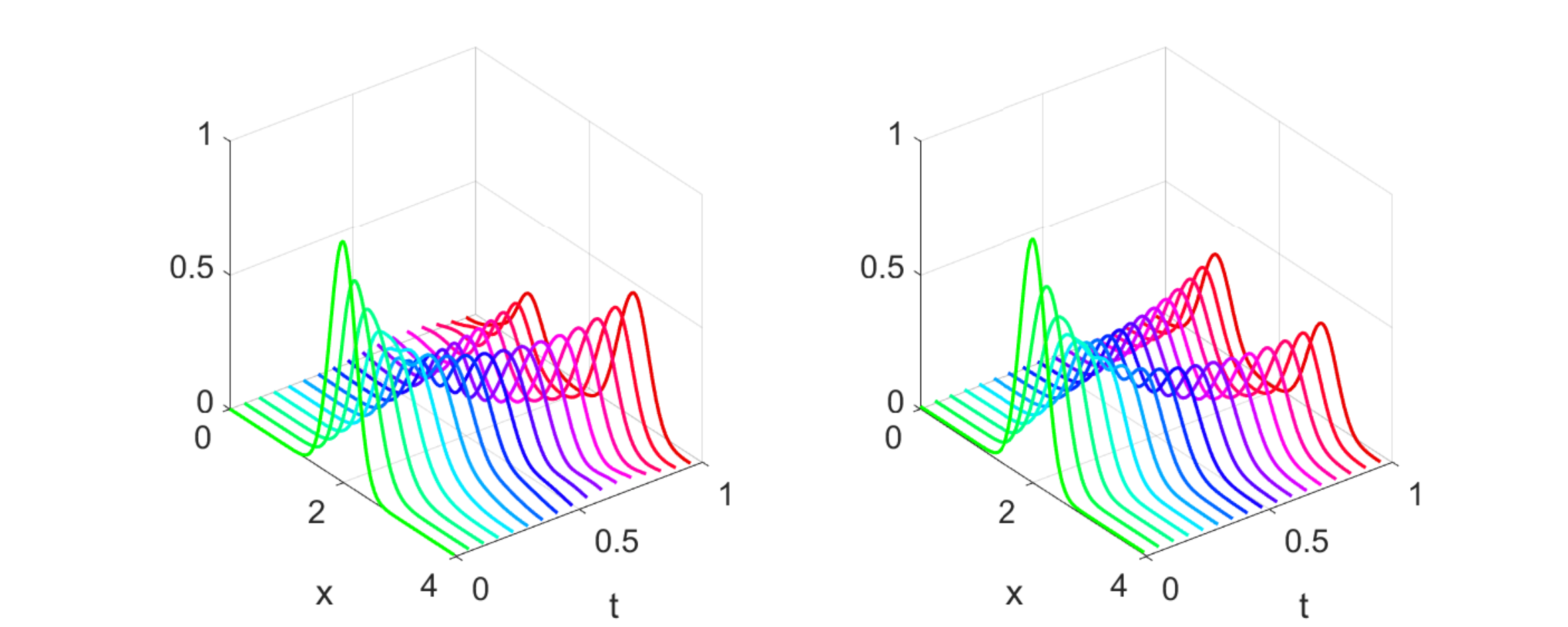}
	\caption{Left: solution $u(t,x)$ for the problem of controlling the conservation law; right: solution $u(t,x)$ for the mean-field game problem.}
	\label{fig:eg2_comparison}
\end{figure}

\subsection{Example 4}
Consider the traffic flow equation with $f(u) = u(1-u), \beta= 10^{-3},\mathcal{F} =- \alpha\int_{\Omega} u \log(u)dx, \alpha \geq 0$. The final cost functional $ \mathcal{H}(u(1,\cdot)) = \int_{\Omega}u(1,x) g(x)dx$.
The initial density and final cost function are as follows
\begin{align*}
u_0(x) = \begin{cases}
0.4 \quad &0.5\leq x \leq 1.5\\
10^{-3} \quad &\text{else}\\
\end{cases},
\quad \quad	g(x) = -0.1\sin(2\pi x).
\end{align*}
The term $\mathcal{F} =-\alpha\int_{\Omega}  u \log(u) dx, \alpha \geq 0$ penalizes the density for getting too concentrate. 
Figure \ref{fig:eg4} shows the effect of the term $\alpha  u \log(u)$, where the density is more spreading in space for the case $\alpha = 1$ than the case $\alpha =0.5$ and $\alpha =0$. We can also see from the $u$ profile at the terminal time. In Figure \ref{fig:eg4_H} (middle), $\alpha = 0$ case has $u(1,\cdot)$ has the most concentrated densities than others. We also numerically verify that the Hamiltonian functional is preserved over time in this example. In Figure \ref{fig:eg4_H} (right), the numerical Hamiltonian ${H}_{\mathcal{G}}(u,\Phi)$ is preserved with error of order of $O(\Delta x)$.

\begin{figure}[htbp!]
	\includegraphics[width=1.0\textwidth,trim=20 20 20 40, clip= true]{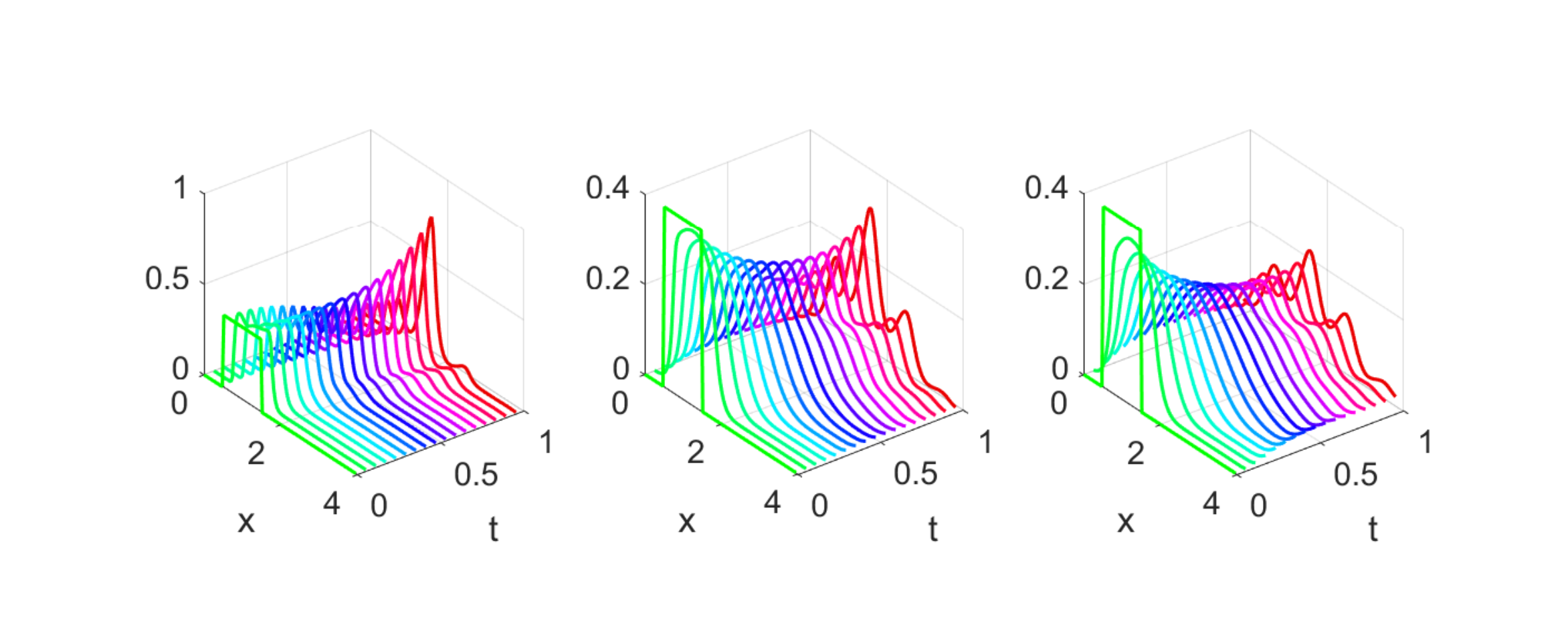}
	\caption{Solution $u(t,x)$ for the problem of controlling the conservation law. From left to right: $\alpha = 0, 0.5,1$.}
	\label{fig:eg4}
\end{figure}

\begin{figure}[htbp!]
	\includegraphics[width=1.0\textwidth,trim=20 20 20 40, clip=true]{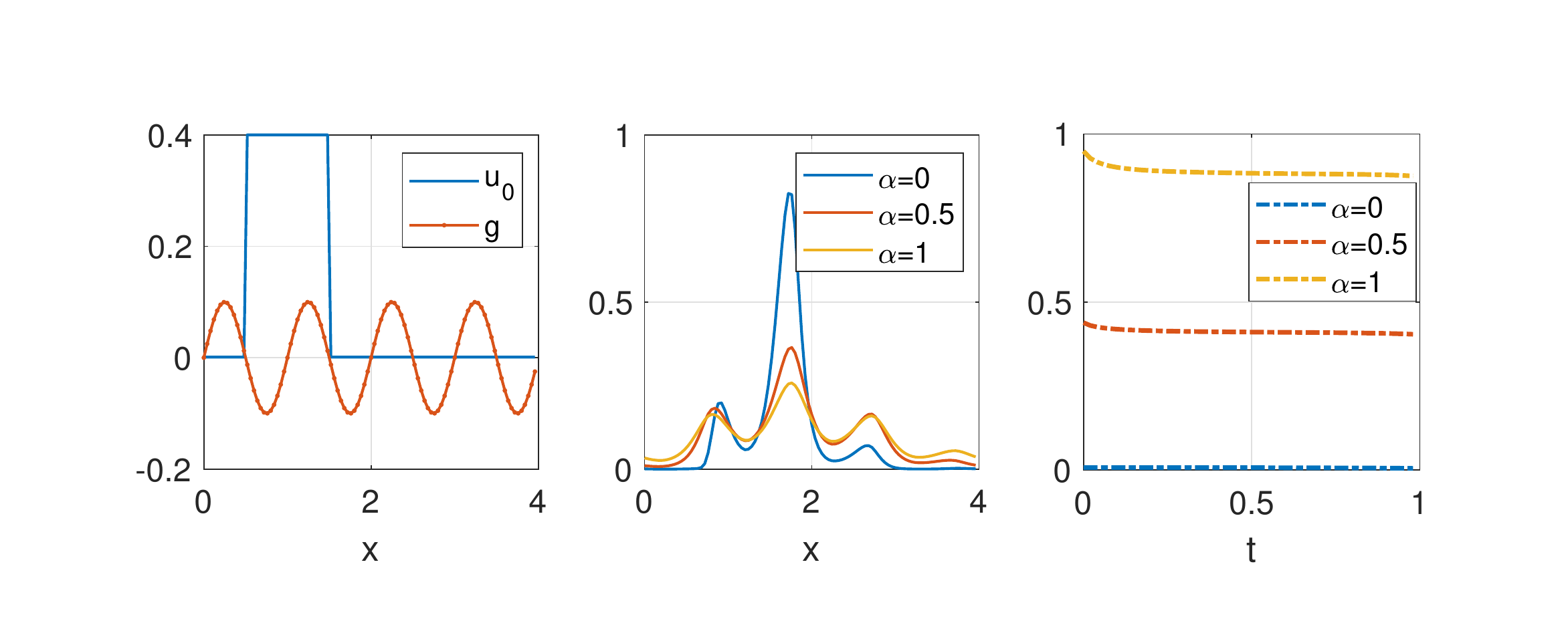}
	\caption{Left: boundary conditions for the control problems $u_0$. Middle: solution $u(1,x)$ for the problem of controlling the conservation law. Right: the numerical Hamiltonian ${H}_{\mathcal{G}}(u,\Phi)$.}
	\label{fig:eg4_H}
\end{figure}

\newpage
\bibliographystyle{abbrv}

\end{document}